\documentclass[11pt]{siamltex}
\usepackage{amsmath}
\usepackage{bm}
\usepackage{amssymb,version}
\usepackage{cases}
\usepackage{color}
\usepackage{verbatim}

\allowdisplaybreaks
\begin{document}
    \title{New SAV-pressure correction  methods  for the Navier-Stokes equations: stability and error analysis
\thanks{The work of X. Li is supported by the National Natural Science Foundation of China  grants  11901489, 11971407 and Postdoctoral Science Foundation of China under grant numbers BX20190187 and 2019M650152. The work of J. Shen is supported in part by NSF grant DMS-1720442 and AFOSR  grant FA9550-16-1-0102.}}

    \author{ Xiaoli Li
        \thanks{School of Mathematical Sciences and Fujian Provincial Key Laboratory on Mathematical Modeling and High Performance Scientific Computing, Xiamen University, Xiamen, Fujian, 361005, China. Email: xiaolisdu@163.com}.
        \and Jie Shen 
         \thanks{Corresponding Author. Department of Mathematics, Purdue University, West Lafayette, IN 47907, USA. Email: shen7@purdue.edu}.
         \and Zhengguang Liu
        \thanks{School of Mathematics and Statistics, Shandong Normal University, Jinan, Shandong, 250358, China. Email: liuzhgsdu@yahoo.com}.
}

\maketitle

\begin{abstract}
 We construct new first- and second-order pressure correction schemes using the scalar auxiliary variable (SAV) approach  for the Navier-Stokes equations. These schemes are linear, decoupled and only require a sequence of solving Poisson type equations at each time step. Furthermore, they are unconditionally energy stable. We also establish rigorous error estimates for the velocity and pressure approximation of the first-order scheme without any condition on the time step.  
\end{abstract}

\markboth{XIAOLI LI, JIE SHEN AND ZHENGGUANG LIU} {Error estimates  for the SAV
approach for the Navier-Stokes equations}
 \section{Introduction}
 We consider  numerical approximation of the time-dependent incompressible Navier-Stokes equations
  \begin{subequations}\label{e_model}
    \begin{align}
     \frac{\partial \textbf{u}}{\partial t}+\textbf{u}\cdot \nabla\textbf{u}
     -\nu\Delta\textbf{u}+\nabla p=\textbf{f}
     \quad &\ in\ \Omega\times J,
      \label{e_modelA}\\
      \nabla\cdot\textbf{u}=0
      \quad &\ in\ \Omega\times J,
      \label{e_modelB}\\
     \textbf{u}=\textbf{0} \quad &\ on\ \partial\Omega\times J,
      \label{e_modelC}
    \end{align}
  \end{subequations}
where $\Omega$ is an open bounded domain in $\mathbb{R}^2$ with a sufficiently smooth boundary $\partial \Omega$, $J=(0,T]$, $(\textbf{u},p)$ represent the unknown   velocity and pressure,  $\textbf{f}$ is an external body force,
 $\nu>0$ is the  viscosity coefficient and $\textbf{n}$ is the unit outward normal  of the domain $\Omega$. 
 
 The above system is one of the most fundamental system in mathematical and physical science.
  Its numerical approximations   plays an eminent role in many branches of science and engineering, and an enormous amount of work have been devoted to the design and  analysis of numerical schemes for its approximation,  see, for instance, \cite{girault1979finite,temam2001navier,glowinski2003finite,gunzburger2012finite} and the references therein. 
  
  Two of the main difficulties in numerically solving Navier-Stokes equations are: (i)   the coupling of velocity and pressure by the incompressible condition $\nabla\cdot \textbf{u}=0$; and (ii)  the treatment of nonlinear term. There are essentially two classes of numerical approaches to deal with the incompressible constraint: the coupled approach and the decoupled approach. The coupled approach  requires solving a saddle point problem at each time step so it could be computationally expensive for dynamical simulations although many efficient solution techniques are available \cite{girault1979finite,brezzi2012mixed,elman2014finite}. The decoupled approach, all originated from the so called projection method \cite{chorin1968numerical,temam1969approximation,guermond2006overview}, leads to a sequence of Poisson type equations to solve at each time step, assuming that the nonlinear term is treated explicitly, hence it can be extremely efficient, particularly for dynamical simulations using finite difference or spectral methods.

 From a computational point of view, it is desirable to be able to treat the nonlinear term explicitly so that one only needs to solve simple linear equations with constant coefficients at each time step. This is specially beneficial if a decoupled approach is used so one only needs to solve a sequence of Poisson type equations to solve at each time step. However, such an explicitly treatment usually leads to  a stability constraint on the time step. To the best of the authors' knowledge, apart from the recently developed schemes  \cite{lin2019numerical} based on the scalar auxiliary variable (SAV) approach \cite{shen2018scalar,shen2017new}, there were no schemes with explicit treatment of nonlinear term that were unconditionally energy diminishing, an important property satisfied by the exact solution of the Navier-Stokes equations. 
 We mention however that it is possible to prove that the numerical solution of a semi-implicit scheme remains to be bounded (but not  energy diminishing) assuming the time step is sufficiently small, but independent of spatial discretization size, see for instance \cite{weinan1995projection,HeSu07}. 
In a recent work \cite{lin2019numerical}, Dong et al. constructed  the following scheme: 
Find ($\textbf{u}^{n+1}$, $p^{n+1}$, $q^{n+1}$) by solving
  \begin{numcases}{}
   \frac{{\textbf{u}}^{n+1}-\textbf{u}^{n}}{\Delta t}+\frac{q^{n+1}}{\sqrt{E(\textbf{u}^{n})+C_0}}\textbf{u}^{n}\cdot \nabla\textbf{u}^{n}
     -\nu\Delta{\textbf{u}}^{n+1}     +\nabla p^{n+1}=0, \ \ \textbf{u}^{n+1}|_{\partial \Omega}=0; \label{e_model_semi_original1}\\
    \nabla\cdot\textbf{u}^{n+1}=0,  \label{e_model_semi_original3}\\ 
   2q^{n+1}\frac{q^{n+1}-q^n}{\Delta t}=(\frac{\textbf{u}^{n+1}-\textbf{u}^{n}}{\Delta t}+ 
   \frac{q^{n+1}}{\sqrt{E(\textbf{u}^{n})+C_0}}(\textbf{u}^n\cdot\nabla) \textbf{u}^n,{\textbf{u}}^{n+1}), \label{e_model_semi_original4} 
\end{numcases}
where $E(\textbf{u})=\int_{\Omega}\frac{1}{2}|\textbf{u}|^2$ is the total energy. It is shown in \cite{lin2019numerical} that
the above scheme satisfies the following property:
\begin{equation}\label{energy_stable}
 |q^{n+1}|^2-|q^n|^2\le -\nu\|\nabla \textbf{u}^{n+1}\|_{L^2(\Omega)}^2,\quad \forall n\ge 0.
\end{equation}
Since $q^{n}$ is an approximation of the energy $E(\textbf{u}(t^n))$, the above scheme is unconditionally energy stable with a modified energy.
It can be shown that the above scheme reduces to two generalized Stokes equations (with constant coefficient) plus a nonlinear algebraic equation for the auxiliary variable $q^{n+1}$ at each time step. So the scheme is essentially as efficient as the usual semi-implicit scheme without the auxiliary variable. Moreover, one can also adopt a pressure-correction strategy so that the two generalized Stokes equations at each time step can be replaced by  a sequence of Poisson-type equations.
Ample numerical results presented in \cite{lin2019numerical} shown that the above scheme is  more efficient and robust than the usual semi-implicit schemes. However, there are also some theoretical and practical issues: (i) It only provides a bound for the scalar sequence $\{q^n\}$ which is intended as an approximation of the energy $E(\textbf{u})$ but with no direct relation in the discrete case. 
(ii) The scheme requires solving a nonlinear algebraic equation. 
 Hence, it is very difficult to shown that  the nonlinear algebraic equation always has a  real positive solution and  to derive an error estimate based just on \eqref{energy_stable}. 
 
The main purpose of this paper is to construct  new SAV schemes for the Navier-Stokes equations and to carry out a rigorous error analysis. Our main contribution are:
\begin{itemize} 
\item We  construct  new SAV schemes with first-order pressure-correction and second-order rotational pressure-correction. The new schemes enjoy  the following additional advantages: (i) it is  purely linear so it does not require solving nonlinear algebraic equation; (ii) it provides better stability: instead of \eqref{energy_stable}, our first-order scheme satisfies 
{\color{black}
\begin{equation*}\label{energy_stable2}
(\|\textbf{u}^{n+1}\|_{L^2(\Omega)}^2+ |q^{n+1}|^2)-(\|\textbf{u}^{n}\|_{L^2(\Omega)}^2+|q^n|^2)\le -2 \nu\|\nabla \textbf{u}^{n+1}\|_{L^2(\Omega)}^2,\quad \forall n\ge 0, \end{equation*}
}
where the extra term $\|\textbf{u}^{n+1}\|_{L^2(\Omega)}^2$  is essential  to carry out an error analysis;
(iii) it is coupled with a pressure-correction strategy \cite{van1986second,guermond2004error} so only Poisson-type equations need to be solved at each time step.  
\item 
We prove our new second-order scheme  based on  the second-order rotational pressure-correction is unconditionally energy stable. Note that
the energy stability of second-order rotational pressure-correction schemes has  been proved only for the  time dependent Stokes equations \cite{guermond2004error,CheS20},  its energy stability 
for the Navier-Stokes equations has been open with any kind of treatment for the nonlinear terms.
To the best of our knowledge, these are the first purely linear schemes for Navier-Stokes equations with explicit treatment of nonlinear terms with proven  unconditional energy stability.   
\item We carry out a rigorous error analysis for our first-order scheme and derive optimal error estimates for the velocity and pressure without any restriction on the time step.
\end{itemize}


The paper is organized as follows. In Section 2,  we provide some preliminaries. In Section 3, we present 
first- and second-order pressure correction projection schemes based on the SAV approach, and describe the solution procedure. In Section 4, we derive the unconditional energy stability for both first- and second-order schemes. In Section 5, we carry out a rigorous error analysis to establish  for the first-order SAV pressure-correction scheme. Numerical experiments  are presented in Section 6 to  validate our theoretical results.

  \section{Preliminaries}
We describe below some notations and results which will be frequently used in this paper.

Throughout the paper, we use $C$, with or without subscript, to denote a positive
constant, which could have different values at different appearances.

 Let $\Omega$ be an open bounded domain in $\mathbb{R}^2$, we will use the standard notations $L^2(\Omega)$, $H^k(\Omega)$ and $H^k_0(\Omega)$ to denote the usual Sobolev spaces over $\Omega$. The norm corresponding to $H^k(\Omega)$ will be denoted simply by $\|\cdot\|_k$. In particular, we use $\|\cdot\|$ to denote the norm in $L^2(\Omega)$. Besides, $(\cdot,\cdot)$ is used to denote the inner product in $L^2(\Omega)$. The vector functions and vector spaces will be indicated by boldface type.
  
  We define   
  $$\textbf{H}=\{ \textbf{v}\in \textbf{L}^2(\Omega): div\textbf{v}=0, \textbf{v}\cdot \textbf{n}|_{\Gamma}=0 \},\ \ \textbf{V}=\{\textbf{v}\in H^1_0(\Omega):  div\textbf{v}=0 \},$$
  and the Stokes operator
  $$ A\textbf{u}=-P_{H}\Delta\textbf{u},\ \ \forall \ \textbf{u}\in D(A)=\textbf{H}^2(\Omega)\cap\textbf{V},$$
where $P_{H}$ is the orthogonal projector in $\textbf{L}^2(\Omega)$ onto $\textbf{H}$ and the Stokes operator $A$ is an unbounded positive self-adjoint closed operator in $\textbf{H}$ with domain $D(A)$.

 Let us recall the following inequalities which will be used in the sequel \cite{temam2001navier,heywood1982finite}: 
\begin{equation}\label{e_norm H2}
\aligned
\|\nabla\textbf{v}\|\leq c_1\|A^{\frac{1}{2}}\textbf{v}\|,\ \ \|\Delta\textbf{v}\|\leq c_1\|A\textbf{v}\|, \ \ \forall \textbf{v}\in D(A)=\textbf{H}^2(\Omega)\cap\textbf{V}.
\endaligned
\end{equation} 
We then derive from the above and  Poincar\'e inequality that
\begin{equation}\label{e_norm H1}
\aligned
\|\textbf{v}\|\leq c_1\|\nabla\textbf{v}\|, \ \forall\textbf{v}\in \textbf{H}^1_0(\Omega),\ \ 
\|\nabla\textbf{v}\|\leq c_1\|A\textbf{v}\|, \ \ \forall \textbf{v}\in D(A),
\endaligned
\end{equation}
where $c_1$ is a positive constant depending only on $\Omega$.

Next we define the trilinear form $b(\cdot,\cdot,\cdot)$ by
\begin{equation*}
\aligned
b(\textbf{u},\textbf{v},\textbf{w})=\int_{\Omega}(\textbf{u}\cdot\nabla)\textbf{v}\cdot \textbf{w}d\textbf{x}.
\endaligned
\end{equation*}
We can easily obtain that the trilinear form $b(\cdot,\cdot,\cdot)$ is a skew-symmetric with respect to its last two arguments, i.e., 
\begin{equation}\label{e_skew-symmetric1}
\aligned
b(\textbf{u},\textbf{v},\textbf{w})=-b(\textbf{u},\textbf{w},\textbf{v}),\ \ \forall \textbf{u}\in \textbf{H}, \ \ \textbf{v}, \textbf{w}\in \textbf{H}^1_0(\Omega),
\endaligned
\end{equation}
and 
\begin{equation}\label{e_skew-symmetric2}
\aligned
b(\textbf{u},\textbf{v},\textbf{v})=0,\ \ \forall \textbf{u}\in \textbf{H}, \ \ \textbf{v}\in \textbf{H}^1_0(\Omega).
\endaligned
\end{equation}
By using a combination of integration by parts, Holder's inequality, and Sobolev inequalities, we can obtain that for $d=2,3$, we have \cite{Tema95,shen1992error}
\begin{flalign}\label{e_estimate for trilinear form}
b(\textbf{u},\textbf{v},\textbf{w})\leq \left\{
   \begin{array}{l}
   c_2\|\textbf{u}\|_1\|\textbf{v}\|_1\|\textbf{w}\|_1,\ \ \forall \textbf{u}, \textbf{v} \in \textbf{H}
   , \textbf{w}\in \textbf{H}^1_0(\Omega),\\
   c_2\|\textbf{u}\|_2\|\textbf{v}\|\|\textbf{w}\|_1, \ \ \forall \textbf{u}\in \textbf{H}^2(\Omega)\cap\textbf{H},\ \textbf{v} \in \textbf{H}, \textbf{w}\in \textbf{H}^1_0(\Omega),\\
   c_2\|\textbf{u}\|_2\|\textbf{v}\|_1\|\textbf{w}\|, \ \ \forall \textbf{u}\in \textbf{H}^2(\Omega)\cap\textbf{H},\ \textbf{v} \in \textbf{H}, \textbf{w}\in \textbf{H}^1_0(\Omega),\\
   c_2\|\textbf{u}\|_1\|\textbf{v}\|_2\|\textbf{w}\|, \ \ \forall \textbf{v}\in \textbf{H}^2(\Omega)\cap\textbf{H},\ \textbf{u}\in \textbf{H}, \textbf{w}\in \textbf{H}^1_0(\Omega),\\
   c_2\|\textbf{u}\|\|\textbf{v}\|_2\|\textbf{w}\|_1, \ \ \forall \textbf{v}\in \textbf{H}^2(\Omega)\cap\textbf{H},\ \textbf{u} \in \textbf{H}, \textbf{w} \in \textbf{H}^1_0(\Omega),\\
   c_2\|\textbf{u}\|_1^{1/2}\|\textbf{u}\|^{1/2}\|\textbf{v}\|_1^{1/2}\|\textbf{v}\|^{1/2}\|\textbf{w}\|_1,\ \ \forall \textbf{u}, \textbf{v} \in \textbf{H},  \textbf{w}\in \textbf{H}^1_0(\Omega),
   \end{array}
   \right.
\end{flalign}
where $c_2$ is a positive constant depending only on $\Omega$. 

We will frequently use the following discrete version of the Gronwall lemma \cite{shen1990long,HeSu07}:

\medskip
\begin{lemma} \label{lem: gronwall2}
Let $a_k$, $b_k$, $c_k$, $d_k$, $\gamma_k$, $\Delta t_k$ be nonnegative real numbers such that
\begin{equation}\label{e_Gronwall3}
\aligned
a_{k+1}-a_k+b_{k+1}\Delta t_{k+1}+c_{k+1}\Delta t_{k+1}-c_k\Delta t_k\leq a_kd_k\Delta t_k+\gamma_{k+1}\Delta t_{k+1}
\endaligned
\end{equation}
for all $0\leq k\leq m$. Then
 \begin{equation}\label{e_Gronwall4}
\aligned
a_{m+1}+\sum_{k=0}^{m+1}b_k\Delta t_k \leq \exp \left(\sum_{k=0}^md_k\Delta t_k \right)\{a_0+(b_0+c_0)\Delta t_0+\sum_{k=1}^{m+1}\gamma_k\Delta t_k \}.
\endaligned
\end{equation}
\end{lemma}

  \section{The pressure-correction schemes based on the SAV approach}
 In this section, we construct the first- and second-order pressure-correction schemes based on the SAV approach for the Navier-Stokes equations.
 
Set $$\Delta t=T/N,\ t^n=n\Delta t, \ d_t g^{n+1}=\frac{g^{n+1}-g^n}{\Delta t},
\ {\rm for} \ n\leq N,$$
 and define a scalar function
 \begin{equation}\label{e_definition of q}
\aligned
q(t)=\rm{exp} (-\frac{t}{T}).
\endaligned
\end{equation} 
This function will serve as the scalar auxiliary variable (SAV).
Then, we rewrite the governing system into the following equivalent form:
  \begin{numcases}{}
 \frac{\partial \textbf{u}}{\partial t}+\frac{q(t)}{\rm{exp}(-\frac{t}{T} )}\textbf{u}\cdot \nabla\textbf{u}
     -\nu\Delta\textbf{u}+\nabla p=\textbf{f},  \label{e_model_transform1}\\
  \frac{\rm{d} q}{\rm{d} t}=-\frac{1}{T}q+\frac1{\rm{exp}(-\frac{t}{T} )}\int_{\Omega}\textbf{u}\cdot \nabla\textbf{u}\cdot \textbf{u}d\textbf{x},     \label{e_model_transform2}\\
 \nabla\cdot\textbf{u}=0. \label{e_model_transform3}
\end{numcases}
Note that the last term in \eqref{e_model_transform2} is zero thanks to \eqref{e_skew-symmetric2}.  {\color{black} This term is added to balance the nonlinear term in  \eqref{e_model_transform1} } in the discretized case. It is clear that the above system is equivalent to the original system.  
We construct below linear, decoupled, first-order and second-order pressure-correction schemes for the above system.

\textbf{Scheme \uppercase\expandafter{\romannumeral 1} (first-order accuracy):} The first-order semi-discrete version of the pressure-correction method can be written as follows: Find ($\tilde{\textbf{u}}^{n+1}, \textbf{u}^{n+1}, p^{n+1}, q^{n+1}$) by solving
  \begin{equation}
   \frac{\tilde{\textbf{u}}^{n+1}-\textbf{u}^{n}}{\Delta t}+\frac{q^{n+1}}{\rm{exp}( -\frac{t^{n+1}}{T} )}\textbf{u}^{n}\cdot \nabla\textbf{u}^{n}
     -\nu\Delta\tilde{\textbf{u}}^{n+1}     +\nabla p^{n}=\textbf{f}^{n+1}, \ \ \tilde{\textbf{u}}^{n+1}|_{\partial \Omega}=0; \label{e_model_semi1}
     \end{equation}
     \begin{eqnarray}
  &&   \frac{\textbf{u}^{n+1}-\tilde{\textbf{u}}^{n+1}}{\Delta t}+\nabla(p^{n+1}-p^n)=0; \label{e_model_semi2}\\
 &&    \nabla\cdot\textbf{u}^{n+1}=0, \ \ \textbf{u}^{n+1}\cdot \textbf{n}|_{\partial \Omega}=0;
 \label{e_model_semi3}
 \end{eqnarray}
 \begin{equation}
   \frac{q^{n+1}-q^n}{\Delta t}=-\frac{1}{T}q^{n+1}+\frac{1}{\rm{exp}(  -\frac{t^{n+1}}{T} )}
(\textbf{u}^n\cdot\nabla \textbf{u}^n,\tilde{\textbf{u}}^{n+1}). \label{e_model_semi4} 
\end{equation}
 We now describe how to solve the semi-discrete-in-time scheme
 \eqref{e_model_semi1}-\eqref{e_model_semi4} efficiently.  We denote $S^{n+1}=
\exp ( \frac{t^{n+1}}{T}) q^{n+1} $ and set
  \begin{numcases}{}
  \tilde{\textbf{u}}^{n+1}=\tilde{\textbf{u}}_1^{n+1}+S^{n+1}\tilde{\textbf{u}}_2^{n+1},\label{e_implementation_1}\\
   \textbf{u}^{n+1}=\textbf{u}_1^{n+1}+S^{n+1}\textbf{u}_2^{n+1},\label{e_implementation_2}\\ 
   p^{n+1}=p_1^{n+1}+S^{n+1}p_2^{n+1}. \label{e_implementation_3}
\end{numcases}
 Plugging  \eqref{e_implementation_1} in the scheme \eqref{e_model_semi1}-\eqref{e_model_semi4}, we find  that $\tilde{\textbf{u}}_i^{n+1}$ $(i=1,2)$ 
 satisfy 
   \begin{numcases}{}
\frac{\tilde{\textbf{u}}_1^{n+1}-\textbf{u}^{n}}{\Delta t}=\nu\Delta \tilde{\textbf{u}}_1^{n+1}-\nabla p_{1}^n+\textbf{f}^{n+1}, \ \ \tilde{\textbf{u}}_1^{n+1}|_{\partial \Omega}=0;
\label{e_implementation_4}\\
\frac{\tilde{\textbf{u}}_2^{n+1}}{\Delta t}+\textbf{u}^{n}\cdot \nabla\textbf{u}^{n}=\nu\Delta \tilde{\textbf{u}}_2^{n+1}-\nabla p_2^n,  \ \ \ \tilde{\textbf{u}}_2^{n+1}|_{\partial \Omega}=0.
\label{e_implementation_5}
\end{numcases}
Then, we can determine $\textbf{u}_{i}^{n+1}$, $p_i^{n+1} $ $(i=1,2)$ by 
   \begin{numcases}{}
\frac{\textbf{u}_i^{n+1}-\tilde{\textbf{u}}_i^{n+1}}{\Delta t}+\nabla (p_i^{n+1}-p_i^n)=0,
\label{e_implementation_6}\\
\nabla\cdot \textbf{u}_i^{n+1}=0,\ \ \ \textbf{u}_i^{n+1}\cdot \textbf{n}|_{\partial \Omega}=0.
\label{e_implementation_7}
\end{numcases} 
 Once $\tilde{\textbf{u}}_i^{n+1}$, $\textbf{u}_{i}^{n+1}$, $p_i^{n+1} $ $(i=1,2)$ are known, we can determine explicitly  $S^{n+1}$ from \eqref{e_model_semi4} as follows:
{\color{black}
 \begin{equation} \label{e_S_solve}
 \aligned
 &\left( \frac{T+\Delta t}{T \Delta t }- \exp( \frac{2 t^{n+1} }{T} ) (\textbf{u}^n\cdot\nabla \textbf{u}^n,\tilde{\textbf{u}}_2^{n+1})  \right) \exp( -\frac{t^{n+1} }{T} )S^{n+1} \\
& \ \ \ \ \ \ 
 =\exp( \frac{t^{n+1} }{T} ) (\textbf{u}^n\cdot\nabla \textbf{u}^n,\tilde{\textbf{u}}_1^{n+1}) +\frac{1}{\Delta t} q^n.
 \endaligned
\end{equation}
}
Finally,  we can obtain $\textbf{u}^{n+1}$ and $p^{n+1}$ from \eqref{e_implementation_2}-\eqref{e_implementation_3}. 

In summary, at each time step, we only need to solve two Poisson-type equations \eqref{e_implementation_4}-\eqref{e_implementation_5}, and 
\eqref{e_implementation_6}-\eqref{e_implementation_7} which can be solved as two Poisson equations for $p_i^{n+1}-p^n_i\, (i=1,2)$ with homogeneous Neumann boundary conditions. Hence, the scheme is very efficient.

\textbf{Scheme \uppercase\expandafter{\romannumeral 2} (second-order accuracy):} The second-order semi-discrete version of the rotational pressure-correction method \cite{guermond2004error} can be written as follows: Find ($\tilde{\textbf{u}}^{n+1}, \textbf{u}^{n+1}, p^{n+1}, q^{n+1}$) by solving
  \begin{equation}
   \frac{3\tilde{\textbf{u}}^{n+1}-4\textbf{u}^{n}+\textbf{u}^{n-1} }{2\Delta t}+\frac{q^{n+1}}{\rm{exp}( -\frac{t^{n+1}}{T} )}\bar{\textbf{u}}^{n}\cdot \nabla \bar{ \textbf{u} }^{n}
     -\nu\Delta\tilde{\textbf{u}}^{n+1}     +\nabla p^{n}=\textbf{f}^{n+1},  \tilde{\textbf{u}}^{n+1}|_{\partial \Omega}=0; \label{e_model_semi_second1}
  \end{equation}   
   \begin{eqnarray}  
   &&  \frac{3\textbf{u}^{n+1}-3\tilde{\textbf{u}}^{n+1}}{2\Delta t}+\nabla(p^{n+1}-p^n+\nu \nabla \cdot \tilde{\textbf{u}}^{n+1})=0; \label{e_model_semi_second2}\\
 &&    \nabla\cdot\textbf{u}^{n+1}=0, \ \ \textbf{u}^{n+1}\cdot \textbf{n}|_{\partial \Omega}=0;
 \label{e_model_semi_second3}
  \end{eqnarray}
 \begin{equation}
   \frac{3q^{n+1}-4q^n+q^{n-1}}{2\Delta t}=-\frac{1}{T}q^{n+1}+\frac{1}{\rm{exp}(  -\frac{t^{n+1}}{T} )}
(\bar{\textbf{u} }^n\cdot\nabla \bar{\textbf{u} }^n,\tilde{\textbf{u}}^{n+1}),
\label{e_model_semi_second4} 
\end{equation}
where $\bar{\textbf{u}}^{n}=2\textbf{u}^{n}-\textbf{u}^{n-1}$. For  $n = 0$, we can compute ($\tilde{\textbf{u}}^{1}$, $\textbf{u}^{1}$, $p^{1}$, $q^{1}$) by the first-order scheme described above.

Implementation of the second-order scheme \eqref{e_model_semi_second1}-\eqref{e_model_semi_second4} is essentially the same as that of the first-order scheme \eqref{e_model_semi1}-\eqref{e_model_semi4}.

  \section{Energy Stability} 
In this section, we will demonstrate that the first- and second-order pressure-correction schemes \eqref{e_model_semi1}-\eqref{e_model_semi4} and \eqref{e_model_semi_second1}-\eqref{e_model_semi_second4} are unconditionally energy stable. 
 \medskip
 
 \begin{theorem}\label{thm_energy stability_first order}
In the absence of the external force $\textbf{f}$, the scheme  \eqref{e_model_semi1}-\eqref{e_model_semi4} is unconditionally stable in the sense that
{\color{black}
\begin{equation*}
\aligned
E^{n+1}-E^{n}\leq -2\nu\Delta t\|\nabla\tilde{{\bf u}}^{n+1}\|^2, \ \ \forall \Delta t,\; n\geq 0,
\endaligned
\end{equation*} 
where 
\begin{equation*}
E^{n+1}=
\|\rm {\bf u}^{n+1}\|^2+|q^{n+1}|^2+(\Delta t)^2\|\nabla p^{n+1}\|^2.
\end{equation*} 
}
\end{theorem}

\begin{proof} 
Taking the inner product of \eqref{e_model_semi1} with   $\Delta t\tilde{\textbf{u}}^{n+1}$ and using the identity 
\begin{equation}\label{e_identity}
\aligned
(a-b,a)=\frac{1}{2}(|a|^2-|b|^2+|a-b|^2),
\endaligned
\end{equation} 
we have
\begin{equation}\label{e_energy decay1}
\aligned
&\frac{\|\tilde{\textbf{u}}^{n+1}\|^2-\|\textbf{u}^{n}\|^2}{2}+\frac{\|\tilde{\textbf{u}}^{n+1}-\textbf{u}^{n}\|^2}{2}
+\Delta t\frac{q^{n+1}}{\rm{exp}( -\frac{t^{n+1}}{T} )}(\textbf{u}^{n}\cdot \nabla\textbf{u}^{n},\tilde{\textbf{u}}^{n+1})\\
&\ \ \ \ \ \ 
=-\nu\Delta t\|\nabla\tilde{\textbf{u}}^{n+1}\|^2-\Delta t(\nabla p^n,\tilde{\textbf{u}}^{n+1}).
\endaligned
\end{equation} 
Recalling \eqref{e_model_semi2}, we have 
\begin{equation}\label{e_energy decay2}
\aligned
\textbf{u}^{n+1}+\Delta t\nabla p^{n+1}=\tilde{\textbf{u}}^{n+1}+\Delta t \nabla p^n.
\endaligned
\end{equation}
Taking the inner product of \eqref{e_energy decay2} with itself on both sides and noticing that $(\nabla p^{n+1},\textbf{u}^{n+1})=-(p^{n+1},\nabla\cdot \textbf{u}^{n+1})=0$, we have
\begin{equation}\label{e_energy decay3}
\aligned
\|\textbf{u}^{n+1}\|^2+(\Delta t)^2\|\nabla p^{n+1}\|^2=\|\tilde{\textbf{u}}^{n+1}\|^2+2\Delta t(\nabla p^n,\tilde{\textbf{u}}^{n+1})+(\Delta t)^2\|\nabla p^n\|^2.
\endaligned
\end{equation}
Combining \eqref{e_energy decay1} with \eqref{e_energy decay3} leads to
\begin{equation}\label{e_energy decay4}
\aligned
&\frac{\|\textbf{u}^{n+1}\|^2-\|\textbf{u}^{n}\|^2}{2}+\frac{\|\tilde{\textbf{u}}^{n+1}-\textbf{u}^{n}\|^2}{2}+\frac{(\Delta t)^2}{2}\|\nabla p^{n+1}\|^2\\
&+\Delta t\frac{q^{n+1}}{\rm{exp}(  -\frac{t^{n+1}}{T} )}(\textbf{u}^{n}\cdot \nabla\textbf{u}^{n},\tilde{\textbf{u}}^{n+1})\\
=&\frac{(\Delta t)^2}{2}\|\nabla p^n\|^2-\nu\Delta t\|\nabla\tilde{\textbf{u}}^{n+1}\|^2.
\endaligned
\end{equation} 
Multiplying \eqref{e_model_semi4} by $q^{n+1}\Delta t$ and using the above equation, we have
{\color{black}
\begin{equation}\label{e_energy decay6}
\aligned
& \frac{1}{2} |q^{n+1}|^2- \frac{1}{2} |q^n|^2+ \frac{1}{2} |q^{n+1}-q^n|^2\\
=&-\frac{1}{T}\Delta t|q^{n+1}|^2+\Delta t\frac{q^{n+1}}{\rm{exp}(  -\frac{t^{n+1}}{T} )}(\textbf{u}^{n}\cdot \nabla\textbf{u}^{n},\tilde{\textbf{u}}^{n+1}).
\endaligned
\end{equation}
}
Then summing up \eqref{e_energy decay4} with \eqref{e_energy decay6} results in 
{\color{black}
\begin{equation*}\label{e_energy decay7}
\aligned
&\|\textbf{u}^{n+1}\|^2-\|\textbf{u}^{n}\|^2+|q^{n+1}|^2-|q^n|^2+\frac{2}{T}\Delta t|q^{n+1}|^2+(\Delta t)^2\|\nabla p^{n+1}\|^2\\
&- (\Delta t)^2 \|\nabla p^n\|^2+|q^{n+1}-q^n|^2+\|\tilde{\textbf{u}}^{n+1}-\textbf{u}^{n}\|^2\\
\hskip 1cm & \leq -2\nu\Delta t\|\nabla\tilde{\textbf{u}}^{n+1}\|^2,
\endaligned
\end{equation*}
}
which implies the desired result.   
\end{proof}

\medskip

The energy stability for any rotational pressure-correction schemes is much more involved \cite{guermond2004error}, particulary in the nonlinear case. Previously, the energy stability of second-order rotational pressure-correction schemes is only proved   for the  time dependent Stokes equations \cite{guermond2004error,CheS20}, and only very recently, an energy stability result is proved for the first-order rotational pressure-correction scheme for the Navier-Stokes equations in \cite{CheS20}. 

 \begin{theorem}\label{thm_energy stability_second order}
In the absence of the external force $\textbf{f}$, the scheme \eqref{e_model_semi_second1}-\eqref{e_model_semi_second4} is unconditionally stable in the sense that 
\begin{equation*}\label{e_energy decay_second}
\aligned
E^{n+1}-E^{n}\leq -2\nu\Delta t\|\nabla\tilde{\textbf{u}}^{n+1}\|^2, \ \ \forall \Delta t,\; n\geq 0,
\endaligned
\end{equation*} 
where 
\begin{equation*}\label{e_definition of E}
\aligned
E^{n+1}= &\| \textbf{u}^{n+1}\|^2+\| 2\textbf{u}^{n+1}-\textbf{u}^{n} \|^2+\frac{4}{3}(\Delta t)^2 \| \nabla (p^{n+1}+g^{n+1})\|^2\\
&+2\nu^{-1}\Delta t\|g^{n+1}\|^2+|q^{n+1}|^2+|2q^{n+1}-q^n|^2,
\endaligned
\end{equation*} 
where $\{g^{n+1}\}$ is recursively  defined by
\begin{equation}\label{e_energy decay_second4}
\aligned
g^0=0,\ \ g^{n+1}=\nu\nabla \cdot \tilde{\textbf{u}}^{n+1}+g^n, \ n\geq 0.
\endaligned
\end{equation}
\end{theorem}

\begin{proof} 
Taking the inner product of \eqref{e_model_semi_second1} with $4\Delta t\tilde{\textbf{u}}^{n+1}$ leads to
\begin{equation}\label{e_energy decay_second1}
\aligned
&2( 3\tilde{\textbf{u}}^{n+1}-4\textbf{u}^{n}+\textbf{u}^{n-1},\tilde{\textbf{u}}^{n+1} )+4\nu\Delta t \| \nabla\tilde{\textbf{u}}^{n+1} \|^2 \\
=&-4\Delta t\frac{q^{n+1}}{\rm{exp}( -\frac{t^{n+1}}{T} )} (\bar{\textbf{u}}^{n}\cdot \nabla \bar{ \textbf{u} }^{n}, \tilde{\textbf{u}}^{n+1} )-4\Delta t(\nabla p^{n}, \tilde{\textbf{u}}^{n+1}).
\endaligned
\end{equation} 
Using \eqref{e_model_semi_second2} and the identity
\begin{equation}\label{e_energy decay_second2}
\aligned
2(3a-4b+c,a)=|a|^2+|2a-b|^2-|b|^2-|2b-c|^2+|a-2b+c|^2,
\endaligned
\end{equation}
we have
\begin{equation}\label{e_energy decay_second3}
\aligned
&2( 3\tilde{\textbf{u}}^{n+1}-4\textbf{u}^{n}+\textbf{u}^{n-1},\tilde{\textbf{u}}^{n+1} )=
2\left( 3( \tilde{\textbf{u}}^{n+1}-\textbf{u}^{n+1} )+3\textbf{u}^{n+1}-4\textbf{u}^{n}+\textbf{u}^{n-1},\tilde{\textbf{u}}^{n+1} \right) \\
&\ \ \ \ \
=6(\tilde{\textbf{u}}^{n+1}-\textbf{u}^{n+1}, \tilde{\textbf{u}}^{n+1})+2( 3\textbf{u}^{n+1}-4\textbf{u}^{n}+\textbf{u}^{n-1},\textbf{u}^{n+1} )\\
&\ \ \ \ \ \ \ \ 
+2( 3\textbf{u}^{n+1}-4\textbf{u}^{n}+\textbf{u}^{n-1},\tilde{\textbf{u}}^{n+1}-\textbf{u}^{n+1} ) \\
&\ \ \ \ \
=3( \|\tilde{\textbf{u}}^{n+1}\|^2- \| \textbf{u}^{n+1}\|^2+  \|\tilde{\textbf{u}}^{n+1}-\textbf{u}^{n+1}\|^2)+\| \textbf{u}^{n+1}\|^2+\| 2\textbf{u}^{n+1}-\textbf{u}^{n} \|^2\\
&\ \ \ \ \ \ \ \ 
-\| \textbf{u}^{n}\|^2-\| 2\textbf{u}^{n}-\textbf{u}^{n-1} \|^2+\| \textbf{u}^{n+1}-2\textbf{u}^{n}+\textbf{u}^{n-1} \|^2.
\endaligned
\end{equation}

Setting $H^{n+1}=p^{n+1}+g^{n+1}$, we can recast \eqref{e_model_semi_second2} as
\begin{equation}\label{e_energy decay_second5}
\aligned
\sqrt{3}\textbf{u}^{n+1}+\frac{2}{\sqrt{3}}\Delta t \nabla H^{n+1}=\sqrt{3} \tilde{\textbf{u}}^{n+1}+\frac{2}{\sqrt{3}}\Delta t \nabla H^{n}.
\endaligned
\end{equation} 
Taking the inner product of \eqref{e_energy decay_second5} with itself on both sides, we have
\begin{equation}\label{e_energy decay_second6}
\aligned
&3 \|\textbf{u}^{n+1} \|^2+\frac{4}{3}(\Delta t)^2 \| \nabla H^{n+1}\|^2\\
=&3 \| \tilde{\textbf{u}}^{n+1} \|^2+\frac{4}{3}(\Delta t)^2 \| \nabla H^{n}\|^2+4\Delta t(\tilde{\textbf{u}}^{n+1}, \nabla p^n )
+4\Delta t(\tilde{\textbf{u}}^{n+1}, \nabla g^n ).
\endaligned
\end{equation} 
Thanks to \eqref{e_energy decay_second4}, we have 
\begin{equation}\label{e_energy decay_second7}
\aligned
&4\Delta t(\tilde{\textbf{u}}^{n+1}, \nabla g^n )=-4\nu^{-1}\Delta t(g^{n+1}-g^n,  g^n)\\
=&2\nu^{-1}\Delta t( \|g^n\|^2- \|g^{n+1}\|^2+ \|g^{n+1}-g^n\|^2)\\
=&2\nu^{-1}\Delta t\|g^n\|^2- 2\nu^{-1}\Delta t\|g^{n+1}\|^2+ 2\nu\Delta t\| \nabla\cdot \tilde{\textbf{u}}^{n+1}\|^2.
\endaligned
\end{equation} 
Using the identity 
\begin{equation}\label{e_energy decay_second8}
\aligned
\| \nabla \times \textbf{v}\|^2+\| \nabla \cdot \textbf{v}\|^2=\| \nabla \textbf{v}\|^2,\ \ \forall \  \textbf{v} \in \textbf{H}^1_0(\Omega),
\endaligned
\end{equation} 
we have
\begin{equation}\label{e_energy decay_second9}
\aligned
&4\Delta t(\tilde{\textbf{u}}^{n+1}, \nabla g^n )
=2\nu^{-1}\Delta t\|g^n\|^2- 2\nu^{-1}\Delta t\|g^{n+1}\|^2\\
&\ \ \ \ \ \ 
+ 2\nu\Delta t\| \nabla \tilde{\textbf{u}}^{n+1}\|^2-2\nu\Delta t\| \nabla \times \textbf{u}^{n+1}\|^2.
\endaligned
\end{equation} 
Then combining \eqref{e_energy decay_second1} with \eqref{e_energy decay_second2}-\eqref{e_energy decay_second9} results in
\begin{equation}\label{e_energy decay_second10}
\aligned
& \| \textbf{u}^{n+1}\|^2+\| 2\textbf{u}^{n+1}-\textbf{u}^{n} \|^2+\frac{4}{3}(\Delta t)^2 \| \nabla H^{n+1}\|^2+2\nu^{-1}\Delta t\|g^{n+1}\|^2\\
&+3 \|\tilde{\textbf{u}}^{n+1}-\textbf{u}^{n+1}\|^2+2\nu\Delta t \| \nabla\tilde{\textbf{u}}^{n+1} \|^2+2\nu\Delta t\| \nabla \times \textbf{u}^{n+1}\|^2\\
\leq & \| \textbf{u}^{n}\|^2+\| 2\textbf{u}^{n}-\textbf{u}^{n-1} \|^2+\frac{4}{3}(\Delta t)^2 \| \nabla H^{n}\|^2+2\nu^{-1}\Delta t\|g^{n}\|^2\\
&-4\Delta t\frac{q^{n+1}}{\rm{exp}( -\frac{t^{n+1}}{T} )} (\bar{\textbf{u}}^{n}\cdot \nabla \bar{ \textbf{u} }^{n}, \tilde{\textbf{u}}^{n+1} ).
\endaligned
\end{equation} 
Multiplying \eqref{e_model_semi_second4} by $4\Delta tq^{n+1}$ and using \eqref{e_energy decay_second2}, we have
\begin{equation}\label{e_energy decay_second11}
\aligned
&|q^{n+1}|^2+|2q^{n+1}-q^n|^2-|q^n|^2-|2q^{n}-q^{n-1}|^2+|q^{n+1}-2q^n+q^{n-1}|^2\\
=&-\frac{4}{T}\Delta t|q^{n+1}|^2+4\Delta t\frac{q^{n+1}}{\rm{exp}(  -\frac{t^{n+1}}{T} )}(\bar{\textbf{u}}^{n}\cdot \nabla \bar{\textbf{u}}^{n},\tilde{\textbf{u}}^{n+1}).
\endaligned
\end{equation}
Then summing up \eqref{e_energy decay_second10} with \eqref{e_energy decay_second11} results in 
\begin{equation}\label{e_energy decay_second12}
\aligned
& \| \textbf{u}^{n+1}\|^2+\| 2\textbf{u}^{n+1}-\textbf{u}^{n} \|^2+\frac{4}{3}(\Delta t)^2 \| \nabla H^{n+1}\|^2+2\nu^{-1}\Delta t\|g^{n+1}\|^2\\
&+|q^{n+1}|^2+|2q^{n+1}-q^n|^2+|q^{n+1}-2q^n+q^{n-1}|^2+\frac{4}{T}\Delta t|q^{n+1}|^2\\
&+3 \|\tilde{\textbf{u}}^{n+1}-\textbf{u}^{n+1}\|^2+2\nu\Delta t \| \nabla\tilde{\textbf{u}}^{n+1} \|^2+2\nu\Delta t\| \nabla \times \textbf{u}^{n+1}\|^2\\
\leq & \| \textbf{u}^{n}\|^2+\| 2\textbf{u}^{n}-\textbf{u}^{n-1} \|^2+\frac{4}{3}(\Delta t)^2 \| \nabla H^{n}\|^2\\
&+2\nu^{-1}\Delta t\|g^{n}\|^2+|q^n|^2+|2q^{n}-q^{n-1}|^2,
\endaligned
\end{equation} 
which implies the desired result.
\end{proof}

  \section{Error Analysis} 
In this section, we carry out a rigorous error analysis for the first-order semi-discrete scheme \eqref{e_model_semi1}-\eqref{e_model_semi4}.

There exist a large body of work devoted to the error analysis of various numerical schemes for the Navier-Stokes equations \eqref{e_model}, we refer to, e.g., 
\cite{heywood1990finite,ammi1994nonlinear,tone2004error,HeSu07,labovsky2009stabilized} for different schemes with coupled approach, and   \cite{shen1992error,weinan1995projection,shen1996error,guermond1999resultat,wang2000convergence,guermond2006overview} for different schemes with decoupled approach. On the other hand, for the SAV approach, some error analysis has been carried out for various gradient flows \cite{shen2018convergence,LSR19,ALL19}. 
 In a recent attempt \cite{li2019error}, we considered a MAC discretization to a second-order version of the scheme \eqref{e_model_semi_original1}-\eqref{e_model_semi_original4} and proved  corresponding error estimates. However, due to the difficulty associated with the nonlinear algebraic equation, we had to assume that  there is a numerical solution satisfying  $q^{n+1}/E(\textbf{u}^n)\ge c_0>0$. Since our new scheme is purely linear, we shall   prove optimal error estimates below without any assumption on the numerical solution.

Let ($\tilde{\textbf{u}}^{n+1}$, $\textbf{u}^{n+1}$, $p^{n+1}$, $q^{n+1}$) be the solution of \eqref{e_model_semi1}-\eqref{e_model_semi4}.  Then we derive immediately from Theorem \ref{thm_energy stability_first order} that 
\begin{equation}\label{e_L2 norm of u and q}
\aligned
\|{\bf u}^{m+1}\|\leq k_0,\ \ |q^{m+1}|\leq k_1, \ \ \forall \ 0\leq m\leq N-1.
\endaligned
\end{equation} 
\begin{equation}\label{e_L2H1 norm of u}
\aligned
 \Delta t\sum_{n=0}^{m}\|{\tilde {\bf u}}^{n+1}\|_{1}^2\leq k_2, \ \ \forall \ 0\leq m\leq N-1,
\endaligned
\end{equation} 
where the constants $k_i$ $(i=0,1,2)$ are independent of $\Delta t$.

We  set
   \begin{numcases}{}
\displaystyle \tilde{e}_{\textbf{u}}^{n+1}=\tilde{\textbf{u}}^{n+1}-\textbf{u}(t^{n+1}),\ \ 
\displaystyle e_{\textbf{u}}^{n+1}=\textbf{u}^{n+1}-\textbf{u}(t^{n+1}), \notag\\
\displaystyle e_{p}^{n+1}=p^{n+1}-p(t^{n+1}),\ \ \ 
\displaystyle e_{q}^{n+1}=q^{n+1}-q(t^{n+1}).\notag
\end{numcases}
\subsection{Error estimates for the velocity}
The main result of this section is stated in the following  theorem.

\begin{theorem}\label{the: error_estimate_final}
Assuming $\textbf{u}\in H^3(0,T;\textbf{L}^2(\Omega))\bigcap H^1(0,T;\textbf{H}^2_0(\Omega))\bigcap W^{1,\infty}(0,T;W^{1,\infty}(\Omega))$, $p\in H^2(0,T;H^1(\Omega))$, then for the first-order  scheme \eqref{e_model_semi1}-\eqref{e_model_semi4}, we have
\begin{equation}\label{e_error_estimate_u_final}
\aligned
&\|e_{\textbf{u}}^{m+1}\|^2+\Delta t\sum\limits_{n=0}^m\|\nabla\tilde{e}_{\textbf{u}}^{n+1}\|^2+\sum\limits_{n=0}^m \|\tilde{e}_{\textbf{u}}^{n+1}-e_{\textbf{u}}^n \|^2+(\Delta t)^2 \|\nabla e_p^{m+1}\|^2 \\
& \ \ \ \ \ \ \ \ 
+|e_q^{m+1}|^2+\Delta t\sum\limits_{n=0}^{m}| d_te_q^{n+1}|^2 \leq C(\Delta t)^2,  \ \ \ \forall \ 0\leq m\leq N-1,
\endaligned
\end{equation}
where $C$ is a positive constant  independent of $\Delta t$.
\end{theorem}
\begin{proof}
We shall follow the steps in the stability  proof of  Theorem \ref{thm_energy stability_first order}.

\noindent{\bf Step 1.} We start by establishing  an error equation  corresponding to \eqref{e_energy decay4}.  Let $\textbf{R}_{\textbf{u}}^{n+1}$ be the truncation error defined by 
\begin{equation}\label{e_error2}
\aligned
\textbf{R}_{\textbf{u}}^{n+1}=\frac{\textbf{u}(t^{n+1})-\textbf{u}(t^{n})}{\Delta t}-\frac{\partial \textbf{u}(t^{n+1})}{\partial t}=\frac{1}{\Delta t}\int_{t^n}^{t^{n+1}}(t-t^n)\frac{\partial^2 \textbf{u}}{\partial t^2}dt.
\endaligned
\end{equation}
Subtracting \eqref{e_model_transform1} at $t^{n+1}$ from \eqref{e_model_semi1}, we obtain
\begin{equation}\label{e_error3}
\aligned
&\frac{\tilde{e}_{\textbf{u}}^{n+1}-e_{\textbf{u}}^n}{\Delta t}-\nu\Delta\tilde{e}_{\textbf{u}}^{n+1}
=\frac{q(t^{n+1})}{ \exp(  -\frac{t^{n+1}}{T} ) }(\textbf{u}(t^{n+1})\cdot \nabla)\textbf{u}(t^{n+1})\\
&\ \ \ \ \ 
-\frac{q^{n+1}}{ \exp(  -\frac{t^{n+1}}{T} )  }\textbf{u}^{n}\cdot \nabla\textbf{u}^{n}
-\nabla (p^n-p(t^{n+1}))+\textbf{R}_{\textbf{u}}^{n+1}.
\endaligned
\end{equation}
We obtain from \eqref{e_model_semi2} that
\begin{equation}\label{e_error4}
\aligned
&\frac{e_{\textbf{u}}^{n+1}-\tilde{e}_{\textbf{u}}^{n+1}}{\Delta t}+\nabla(p^{n+1}-p^n)=0.
\endaligned
\end{equation}
Taking the inner product of \eqref{e_error3} with $\tilde{e}_{\textbf{u}}^{n+1}$, we obtain
\begin{equation}\label{e_error5}
\aligned
&\frac{\|\tilde{e}_{\textbf{u}}^{n+1}\|^2-\|e_{\textbf{u}}^n\|^2}{2\Delta t}+\frac{\|\tilde{e}_{\textbf{u}}^{n+1}-e_{\textbf{u}}^n\|^2}{2\Delta t}+\nu\|\nabla\tilde{e}_{\textbf{u}}^{n+1}\|^2\\
=&\left(\frac{q(t^{n+1})}{ \exp(  -\frac{t^{n+1}}{T} ) }(\textbf{u}(t^{n+1})\cdot \nabla)\textbf{u}(t^{n+1})- \frac{q^{n+1}}{ \exp(  -\frac{t^{n+1}}{T} ) }\textbf{u}^{n}\cdot \nabla\textbf{u}^{n},\tilde{e}_{\textbf{u}}^{n+1}\right)\\
&-(\nabla(p^n-p(t^{n+1})),\tilde{e}_{\textbf{u}}^{n+1})+(\textbf{R}_{\textbf{u}}^{n+1},\tilde{e}_{\textbf{u}}^{n+1}).
\endaligned
\end{equation}
Taking the inner product of \eqref{e_error4} with $\frac{e_{\textbf{u}}^{n+1}+\tilde{e}_{\textbf{u}}^{n+1}}{2}$, we derive
\begin{equation}\label{e_error6}
\aligned
&\frac{\|e_{\textbf{u}}^{n+1}\|^2-\|\tilde{e}_{\textbf{u}}^{n+1}\|^2}{2\Delta t}+\frac{1}{2}\left(\nabla(p^{n+1}-p^n),\tilde{e}_{\textbf{u}}^{n+1}\right)=0.
\endaligned
\end{equation}
Adding \eqref{e_error5} and \eqref{e_error6}, we have
\begin{equation}\label{e_error7}
\aligned
&\frac{\|e_{\textbf{u}}^{n+1}\|^2-\|e_{\textbf{u}}^n\|^2}{2\Delta t}+\frac{\|\tilde{e}_{\textbf{u}}^{n+1}-e_{\textbf{u}}^n\|^2}{2\Delta t}+\nu\|\nabla\tilde{e}_{\textbf{u}}^{n+1}\|^2\\
=&\left(\frac{q(t^{n+1})}{ \exp(  -\frac{t^{n+1}}{T} ) }(\textbf{u}(t^{n+1})\cdot \nabla)\textbf{u}(t^{n+1})- \frac{q^{n+1}}{ \exp(  -\frac{t^{n+1}}{T} ) }\textbf{u}^{n}\cdot \nabla\textbf{u}^{n},\tilde{e}_{\textbf{u}}^{n+1}\right)\\
&-\frac{1}{2}\left(\nabla(p^{n+1}+p^n-2p(t^{n+1})),\tilde{e}_{\textbf{u}}^{n+1}\right)+(\textbf{R}_{\textbf{u}}^{n+1},\tilde{e}_{\textbf{u}}^{n+1}).
\endaligned
\end{equation}
For the first term on the right hand side of \eqref{e_error7}, we have
\begin{equation}\label{e_error8}
\aligned
&\left(\frac{q(t^{n+1})}{ \exp(  -\frac{t^{n+1}}{T} ) }(\textbf{u}(t^{n+1})\cdot \nabla)\textbf{u}(t^{n+1})- \frac{q^{n+1}}{ \exp(  -\frac{t^{n+1}}{T} ) }\textbf{u}^{n}\cdot \nabla\textbf{u}^{n},\tilde{e}_{\textbf{u}}^{n+1}\right)\\
=&\frac{q(t^{n+1}) }{ \exp(  -\frac{t^{n+1}}{T} ) }\left( ( \textbf{u}(t^{n+1})-\textbf{u}^{n} )\cdot \nabla \textbf{u}(t^{n+1}),\tilde{e}_{\textbf{u}}^{n+1}\right)\\
&+\frac{q(t^{n+1}) }{ \exp(  -\frac{t^{n+1}}{T} ) }\left(  \textbf{u}^{n}\cdot \nabla (\textbf{u}(t^{n+1})-\textbf{u}^{n} ),\tilde{e}_{\textbf{u}}^{n+1}\right)\\
&-\frac{e_q^{n+1}}{ \exp(  -\frac{t^{n+1}}{T} ) }\left((\textbf{u}^n \cdot \nabla)\textbf{u}^n,\tilde{e}_{\textbf{u}}^{n+1}\right).
\endaligned
\end{equation}
Thanks to \eqref{e_L2 norm of u and q} and  \eqref{e_estimate for trilinear form}, the first term on the right hand side of \eqref{e_error8} can be estimated by
\begin{equation}\label{e_error9}
\aligned
\frac{q(t^{n+1}) }{ \exp(  -\frac{t^{n+1}}{T} ) }&\left( ( \textbf{u}(t^{n+1})-\textbf{u}^{n} )\cdot \nabla \textbf{u}(t^{n+1}),\tilde{e}_{\textbf{u}}^{n+1}\right)\\
\leq &c_2(1+c_1)\exp(1) \| \textbf{u}(t^{n+1})-\textbf{u}^{n} \| \| \textbf{u}(t^{n+1})\|_2
\|\nabla \tilde{e}_{\textbf{u}}^{n+1} \| \\
\leq &\frac{\nu}{6} \|\nabla \tilde{e}_{\textbf{u}}^{n+1} \|^2+C \| \textbf{u}(t^{n+1})\|_2^2\|e_{\textbf{u}}^n\|^2+C \| \textbf{u}(t^{n+1})\|_2^2\Delta t\int_{t^n}^{t^{n+1}}\|\textbf{u}_t\|^2dt.
\endaligned
\end{equation}
Using Cauchy-Schwarz inequality and recalling \eqref{e_L2 norm of u and q}, the second term on the right hand side of \eqref{e_error8} can be bounded by
\begin{equation}\label{e_error10}
\aligned
\frac{ q(t^{n+1}) }{ \exp(  -\frac{t^{n+1}}{T} ) }&\left(  \textbf{u}^{n}\cdot \nabla (\textbf{u}(t^{n+1})-\textbf{u}^{n} ),\tilde{e}_{\textbf{u}}^{n+1}\right)\\
= &\frac{ q(t^{n+1}) }{ \exp(  -\frac{t^{n+1}}{T} ) }\left(  \textbf{u}^{n}\cdot \nabla (\textbf{u}(t^{n+1})-\textbf{u}(t^{n}) ),\tilde{e}_{\textbf{u}}^{n+1}\right)-\frac{ q(t^{n+1}) }{ \exp(  -\frac{t^{n+1}}{T} ) }\left( e_{\textbf{u}}^{n}\cdot \nabla e_{\textbf{u}}^n,\tilde{e}_{\textbf{u}}^{n+1}\right)\\
&+\frac{ q(t^{n+1}) }{ \exp(  -\frac{t^{n+1}}{T} ) }\left(  \textbf{u}(t^n)\cdot \nabla e_{\textbf{u}}^n,\tilde{e}_{\textbf{u}}^{n+1}\right)\\
\leq &c_2(1+c_1)\exp(1) \|\nabla \tilde{e}_{\textbf{u}}^{n+1} \| (\|\textbf{u}^{n}\|  \| \int_{t^n}^{t^{n+1}}\textbf{u}_tdt \|_2+\|e_{\textbf{u}}^n\| \| \textbf{u}(t^n)\|_2)\\
&+c_2(1+c_1)\exp(1) \|e_{\textbf{u}}^n\|^{1/2} \|e_{\textbf{u}}^n\|^{1/2}_1\|e_{\textbf{u}}^n\|^{1/2} \|e_{\textbf{u}}^n\|^{1/2}_1\| \nabla \tilde{e}_{\textbf{u}}^{n+1} \| \\
\leq & \frac{\nu}{6} \|\nabla \tilde{e}_{\textbf{u}}^{n+1} \|^2+C (\| \textbf{u}(t^n)\|_2^2+\|e_{\textbf{u}}^n\|^2_1)  \|e_{\textbf{u}}^n\|^2\\
&+C \Delta t \int_{t^n}^{t^{n+1}}\|\textbf{u}_t \|_2^2dt.
\endaligned
\end{equation}
Next we estimate the second term on the right hand side of \eqref{e_error7}. Recalling \eqref{e_error4}, we have
\begin{equation}\label{e_error12}
\aligned
-\frac{1}{2}&\left(\nabla(p^{n+1}+p^n-2p(t^{n+1})),\tilde{e}_{\textbf{u}}^{n+1}\right)=
-\frac{1}{2}\left(\nabla(e_p^{n+1}+e_p^n-p(t^{n+1})+p(t^n)),
\tilde{e}_{\textbf{u}}^{n+1}\right)\\
=&-\frac{1}{2}\left(\nabla(e_p^{n+1}+e_p^n-p(t^{n+1})+p(t^n)),e_{\textbf{u}}^{n+1}+\Delta t(\nabla(e_p^{n+1}-e_p^n)+\nabla(p(t^{n+1})-p(t^n))\right)\\
=&-\frac{\Delta t}{2}(\|\nabla e_p^{n+1}\|^2-\|\nabla e_p^n\|^2)
-\Delta t (\nabla(p(t^{n+1})-p(t^n)),\nabla e_p^n)\\
&+\frac{\Delta t}{2}\|\nabla(p(t^{n+1})-p(t^n))\|^2\\
\leq &-\frac{\Delta t}{2}(\|\nabla e_p^{n+1}\|^2-\|\nabla e_p^n\|^2)+(\Delta t)^2\|\nabla e_p^n\|^2\\
&+C\left(\Delta t+(\Delta t)^2\right)\int_{t^n}^{t^{n+1}}\|\nabla p_t(t)\|^2dt.
\endaligned
\end{equation}
For the last term on the right hand side of \eqref{e_error7}, we have 
\begin{equation}\label{e_error13}
\aligned
&(\textbf{R}_{\textbf{u}}^{n+1},\tilde{e}_{\textbf{u}}^{n+1})\leq \frac{\nu}{6} \|\nabla \tilde{e}_{\textbf{u}}^{n+1} \|^2+C \Delta t\int_{t^n}^{t^{n+1}}\|\textbf{u}_{tt}\|_{-1}^2dt.
\endaligned
\end{equation}
Finally,  combining \eqref{e_error7} with \eqref{e_error8}-\eqref{e_error13}, we obtain
\begin{equation}\label{e_error_estimate_u_L2}
\aligned
\frac{\|e_{\textbf{u}}^{n+1}\|^2-\|e_{\textbf{u}}^n\|^2}{2\Delta t}&+\frac{\|\tilde{e}_{\textbf{u}}^{n+1}-e_{\textbf{u}}^n\|^2}{2\Delta t}+\frac{\nu} {2} \|\nabla\tilde{e}_{\textbf{u}}^{n+1}\|^2+\frac{\Delta t}{2}(\|\nabla e_p^{n+1}\|^2-\|\nabla e_p^n\|^2)\\
\leq &-\frac{e_q^{n+1}}{ \exp(  -\frac{t^{n+1}}{T} ) }\left((\textbf{u}^n \cdot \nabla)\textbf{u}^n,\tilde{e}_{\textbf{u}}^{n+1}\right) + C (\| \textbf{u}(t^n)\|_2^2+\|e_{\textbf{u}}^n\|^2_1)  \|e_{\textbf{u}}^n\|^2 \\
&+(\Delta t)^2\|\nabla e_p^n\|^2+C \| \textbf{u}(t^{n+1})\|_2^2\Delta t\int_{t^n}^{t^{n+1}}\|\textbf{u}_t\|^2dt \\
&+C \Delta t \int_{t^n}^{t^{n+1}}\|\textbf{u}_t \|_2^2dt+C\Delta t\int_{t^n}^{t^{n+1}}\|\textbf{u}_{tt}\|_{-1}^2dt \\
&+C\left(\Delta t+(\Delta t)^2\right)\int_{t^n}^{t^{n+1}}\|\nabla p_t(t)\|^2dt,  \ \ \ \forall \ 0\leq m\leq N-1.
\endaligned
\end{equation}

\noindent{\bf Step 2.} Note that the first term on the right hand side can not be easily bounded. As in the stability proof, we shall   balanced it with a term from the error equation for $q$ corresponding to \eqref{e_energy decay6}. We proceed as follows.

Subtracting \eqref{e_model_transform2} from \eqref{e_model_semi4} leads to
\begin{equation}\label{e_error15}
\aligned
 &\frac{e_q^{n+1}-e_q^n}{\Delta t}+\frac{1}{T}e_q^{n+1}\\
 =&\frac{1}{\rm{exp}(-\frac{ t^{n+1} } {T})}
\left( (\textbf{u}^n\cdot\nabla \textbf{u}^n,\tilde{\textbf{u}}^{n+1})-
(\textbf{u}(t^{n+1})\cdot \nabla \textbf{u}(t^{n+1}),\textbf{u}(t^{n+1}) ) \right)+\textbf{R}_{q}^{n+1},
\endaligned
\end{equation}
where 
\begin{equation}\label{e_error16}
\aligned
\textbf{R}_{q}^{n+1}=\frac{q(t^{n+1})-q(t^{n})}{\Delta t}-\frac{\partial q(t^{n+1})}{\partial t}=\frac{1}{\Delta t}\int_{t^n}^{t^{n+1}}(t-t^n)\frac{\partial^2 q}{\partial t^2}dt.
\endaligned
\end{equation}
Multiplying both sides of \eqref{e_error15} by $e_q^{n+1}$ yields
\begin{equation}\label{e_error17}
\aligned
 &\frac{|e_q^{n+1}|^2-|e_q^n|^2}{2\Delta t}+\frac{|e_q^{n+1}-e_q^n|^2}{2\Delta t}+\frac{1}{T}|e_q^{n+1}|^2\\
 =& \frac{ e_q^{n+1} }{\rm{exp}( -\frac{ t^{n+1} } {T} )} (\textbf{u}^n\cdot\nabla \textbf{u}^n, \tilde{e}_{\textbf{u}}^{n+1} )-\frac{ e_q^{n+1} }{\rm{exp}( -\frac{ t^{n+1} } {T} )} \left( \textbf{u}^n\cdot\nabla ( 
 \textbf{u}(t^{n+1})-\textbf{u}^n), \textbf{u}(t^{n+1}) \right)\\
 &-\frac{ e_q^{n+1} }{\rm{exp}( -\frac{ t^{n+1} } {T} )} \left( ( 
 \textbf{u}(t^{n+1})-\textbf{u}^n)\cdot \nabla \textbf{u}(t^{n+1}), \textbf{u}(t^{n+1}) \right)+\textbf{R}_{q}^{n+1}e_q^{n+1},
\endaligned
\end{equation}
Recalling \eqref{e_skew-symmetric1} and \eqref{e_L2 norm of u and q}, the second term on the right hand side of \eqref{e_error17} can be bounded by
\begin{equation}\label{e_error19}
\aligned
 -\frac{ e_q^{n+1} }{\rm{exp}( -\frac{ t^{n+1} } {T} )} &\left( \textbf{u}^n\cdot\nabla ( 
 \textbf{u}(t^{n+1})-\textbf{u}^n), \textbf{u}(t^{n+1}) \right) \\
 \leq & c_2 \exp(1)\|\textbf{u}^n \|_1 \| \textbf{u}(t^{n+1})-\textbf{u}^n \|_0 \| \textbf{u}(t^{n+1}) \|_{2} |e_q^{n+1}| \\
 \leq & \frac{1}{4k_2} \|\textbf{u}^n \|_1^2 |e_q^{n+1}|^2 + C \| e_{\textbf{u}}^n \|^2
+C\Delta t\int_{t^n}^{t^{n+1} } \|\textbf{u}_t \|_0^2dt,
 \endaligned
\end{equation}
The third term on the right hand side of \eqref{e_error17} can be bounded by
\begin{equation}\label{e_error20}
\aligned
 -\frac{ e_q^{n+1} }{\rm{exp}( -\frac{ t^{n+1} } {T} )}& \left( ( 
 \textbf{u}(t^{n+1})-\textbf{u}^n)\cdot \nabla \textbf{u}(t^{n+1}), \textbf{u}(t^{n+1}) \right)\\
 \leq &c_2\exp(1) \| \textbf{u}(t^{n+1})-\textbf{u}^n \| \|\textbf{u}(t^{n+1}) \|_1 \|\textbf{u}(t^{n+1}) \|_2 |e_q^{n+1}| \\
 \leq & C\|e_{\textbf{u}}^{n}\|^2+\frac{1}{4 T}  |e_q^{n+1}|^2+C\Delta t\int_{t^n}^{t^{n+1} } \|\textbf{u}_t \|^2dt.
\endaligned
\end{equation}
For the last term on the right hand side of \eqref{e_error17}, we have
\begin{equation}\label{e_error21}
\aligned
\textbf{R}_{q}^{n+1}e_q^{n+1} \leq \frac{1}{4 T}  |e_q^{n+1}|^2+C\Delta t\int_{t^n}^{t^{n+1} } \|q_{tt} \|^2dt.
\endaligned
\end{equation}
Combining \eqref{e_error17} with \eqref{e_error19}-\eqref{e_error21} results in 
\begin{equation}\label{e_error22}
\aligned
 \frac{|e_q^{n+1}|^2-|e_q^n|^2}{2\Delta t}&+\frac{|e_q^{n+1}-e_q^n|^2}{2\Delta t}+\frac{1}{2T}|e_q^{n+1}|^2\\
 \leq & \frac{ e_q^{n+1} }{\rm{exp}( -\frac{ t^{n+1} } {T} )} (\textbf{u}^n\cdot\nabla \textbf{u}^n, \tilde{e}_{\textbf{u}}^{n+1} )
  + \frac{1}{4k_2}   \| \textbf{u}^n \|_1^{2}  |e_q^{n+1}|^2 + C\|e_{\textbf{u}}^{n}\|^2 \\
 &+C\Delta t\int_{t^n}^{t^{n+1} } \|\textbf{u}_t \|_0^2dt 
+C\Delta t\int_{t^n}^{t^{n+1} } \|q_{tt} \|^2dt.
\endaligned
\end{equation}
Note that the first term on the right hand side above is what we need to balance the first term on the right hand side of \eqref{e_error_estimate_u_L2}.

\noindent{\bf Step 3}.  Summing up \eqref{e_error22} with \eqref{e_error_estimate_u_L2} leads to
\begin{equation}\label{e_error22_R1}
\aligned
 & \frac{|e_q^{n+1}|^2-|e_q^n|^2}{2\Delta t}+\frac{|e_q^{n+1}-e_q^n|^2}{2\Delta t}+\frac{1}{2T}|e_q^{n+1}|^2 + \frac{\|e_{\textbf{u}}^{n+1}\|^2-\|e_{\textbf{u}}^n\|^2}{2\Delta t} \\
 & +\frac{\|\tilde{e}_{\textbf{u}}^{n+1}-e_{\textbf{u}}^n\|^2}{2\Delta t}+\frac{\nu} {2} \|\nabla\tilde{e}_{\textbf{u}}^{n+1}\|^2+\frac{\Delta t}{2}(\|\nabla e_p^{n+1}\|^2-\|\nabla e_p^n\|^2) \\
\leq & \frac{1}{4k_2}   \| \textbf{u}^n \|_1^{2}  |e_q^{n+1}|^2  + C (\| \textbf{u}(t^n)\|_2^2+\|e_{\textbf{u}}^n\|^2_1)  \|e_{\textbf{u}}^n\|^2 \\
&+(\Delta t)^2\|\nabla e_p^n\|^2+C  \| \textbf{u}(t^{n+1})\|_2^2\Delta t\int_{t^n}^{t^{n+1}}\|\textbf{u}_t\|^2dt \\
&+C \Delta t \int_{t^n}^{t^{n+1}}\|\textbf{u}_t \|_2^2dt+C\Delta t\int_{t^n}^{t^{n+1}}\|\textbf{u}_{tt}\|_{-1}^2dt \\
&+C\left(\Delta t+(\Delta t)^2\right)\int_{t^n}^{t^{n+1}}\|\nabla p_t(t)\|^2dt +C\Delta t\int_{t^n}^{t^{n+1} } \|q_{tt} \|^2dt.
\endaligned
\end{equation}
Multiplying \eqref{e_error22_R1} by $2\Delta t$ and summing over $n$, $n=0,2,\ldots,m^*$, where $m^*$ is the time step at which $|e_q^{m^*+1}|$ achieves its maximum value, we can obtain
\begin{equation}\label{e_error23}
\aligned
  |e_q^{m^*+1}|^2+& \frac{ \Delta t }{T} \sum\limits_{n=0}^{m^*}|e_q^{n+1}|^2 +
  \|e_{\textbf{u}}^{m^*+1}\|^2 + \nu \Delta t  \sum\limits_{n=0}^{m^*}  \|\nabla\tilde{e}_{\textbf{u}}^{n+1}\|^2 + (\Delta t)^2 \|\nabla e_p^{m^*+1}\|^2 \\
  \leq & \frac{1}{2k_2} |e_q^{m^*+1}|^2 \Delta t \sum\limits_{n=0}^{m^*}\| \textbf{u}^n \|_1^{2} +C \Delta t\sum\limits_{n=0}^{m^*} (\| \textbf{u}(t^n)\|_2^2+\|e_{\textbf{u}}^n\|^2_1)  \|e_{\textbf{u}}^n\|^2 \\
  &+(\Delta t)^3 \sum\limits_{n=0}^{m^*} \|\nabla e_p^n\|^2+C \| \textbf{u}(t^{n+1})\|_2^2 (\Delta t) ^2 \int_{t^0}^{t^{m^*+1}}\|\textbf{u}_t\|^2dt \\
&+C (\Delta t)^2 \int_{t^0}^{t^{m^*+1}}\|\textbf{u}_t \|_2^2dt+C (\Delta t)^2 \int_{t^0}^{t^{m^*+1}}\|\textbf{u}_{tt}\|_{-1}^2dt \\
&+C\left( (\Delta t)^2+(\Delta t)^3 \right)\int_{t^0}^{t^{m^*+1}}\|\nabla p_t(t)\|^2dt +C (\Delta t)^2 \int_{t^0}^{t^{m^*+1} } \|q_{tt} \|^2dt.   
\endaligned
\end{equation}
Thanks to  \eqref{e_L2H1 norm of u}, the first term on the right hand side is bounded by $\frac{1}{2} |e_q^{m^*+1}|^2$. Then, applying the discrete Gronwall lemma  \ref{lem: gronwall2}, we obtain
\begin{equation*}
\aligned
 &|e_q^{m^*+1}|^2+\Delta t\sum\limits_{n=0}^{m^*}|e_q^{n+1}|^2  +
  \|e_{\textbf{u}}^{m^*+1}\|^2 + \nu \Delta t  \sum\limits_{n=0}^{m^*}  \|\nabla\tilde{e}_{\textbf{u}}^{n+1}\|^2 + (\Delta t)^2 \|\nabla e_p^{m^*+1}\|^2 \\
&\ \ \ \ \ \ \ \ \ \ 
 \leq  C( \|\textbf{u}_t\|_{H^1(0,T;H^2( {\Omega}) )}^2+\|q\|_{H^2(0.T)}^2 )
  (\Delta t)^2.
\endaligned
\end{equation*}
Since $|e_q^{m^*+1}|=\max_{0\le m\le N-1}|e_q^{m+1}|$, the above  also implies
\begin{equation}\label{e_error24}
\aligned
 &|e_q^{m+1}|^2+\Delta t\sum\limits_{n=0}^{m}|e_q^{n+1}|^2  +
  \|e_{\textbf{u}}^{m+1}\|^2 + \nu \Delta t  \sum\limits_{n=0}^{m}  \|\nabla\tilde{e}_{\textbf{u}}^{n+1}\|^2 \\
& \ \ \ \ \ \  \leq  C( \|\textbf{u}_t\|_{H^1(0,T;H^2( {\Omega} ) )}^2+\|q\|_{H^2(0,T)}^2 )
  (\Delta t)^2, \quad \forall 0\le m\le N-1.
\endaligned
\end{equation}

Next multiplying both sides of \eqref{e_error15} with $d_te_q^{n+1}$ leads to
\begin{equation}\label{e_error_dtq1}
\aligned
 | d_te_q^{n+1}|^2&+ \frac{|e_q^{n+1}|^2-|e_q^n|^2}{2T\Delta t}+\frac{|e_q^{n+1}-e_q^n|^2}{2T\Delta t}\\
 =&\frac{ d_te_q^{n+1} }{\rm{exp}( -\frac{ t^{n+1} } {T} )} (\textbf{u}^n\cdot\nabla \textbf{u}^n, \tilde{e}_{\textbf{u}}^{n+1} )-\frac{ d_te_q^{n+1} }{\rm{exp}( -\frac{ t^{n+1} } {T} )} \left( \textbf{u}^n\cdot\nabla ( 
 \textbf{u}(t^{n+1})-\textbf{u}^n), \textbf{u}(t^{n+1}) \right)\\
 &-\frac{ d_te_q^{n+1} }{\rm{exp}( -\frac{ t^{n+1} } {T} )} \left( ( 
 \textbf{u}(t^{n+1})-\textbf{u}^n)\cdot \nabla \textbf{u}(t^{n+1}), \textbf{u}(t^{n+1}) \right)+\textbf{R}_{q}^{n+1}d_te_q^{n+1}.
\endaligned
\end{equation}
Thanks to \eqref{e_error24}, we have
\begin{equation}\label{e_error_dtq2}
\aligned
&\Delta t\sum\limits_{n=0}^m\|\nabla\tilde{e}_{\textbf{u}}^{n+1}\|^2 \leq C(\Delta t)^2,
\endaligned
\end{equation}
which implies that 
\begin{equation}\label{e_error_dtq3}
\aligned
&\| \textbf{u}^{n+1} \|_{1} \leq C\| \tilde{\textbf{u}}^{n+1} \|_{1} \leq C(\Delta t)^{1/2} \leq k_3.
\endaligned
\end{equation}
The above inequality holds thanks to the fact that \cite{temam2001navier} 
 $$ \| \textbf{u} ^{n+1} \|_{\textbf{H}^1(\Omega)}=
 \|P_{H} \tilde{ \textbf{u} }^{n+1} \|_{\textbf{H}^1(\Omega)} \leq c(\Omega)  \| \tilde{ \textbf{u} }^{n+1} \|_{\textbf{H}^1(\Omega)}.$$
Then the first term on the right hand side of \eqref{e_error_dtq1} can be estimated by 
\begin{equation}\label{e_error_dtq4}
\aligned
\frac{ d_te_q^{n+1} }{\rm{exp}( -\frac{ t^{n+1} } {T} )} &(\textbf{u}^n\cdot\nabla \textbf{u}^n, \tilde{e}_{\textbf{u}}^{n+1} ) \\
\leq &(1+c_1)c_2 \| \textbf{u}^n \|^{1/2}\| \textbf{u}^n \|_1^{1/2} \|\textbf{u}^n \|^{1/2}\| \textbf{u}^n \|_1^{1/2} \| \nabla \tilde{e}_{\textbf{u}}^{n+1} \| |d_te_q^{n+1}| \\
\leq &\frac{1}{6}|d_te_q^{n+1}|^2+C \| \nabla \tilde{e}_{\textbf{u}}^{n+1} \|^2.
\endaligned
\end{equation}
The second and third terms on the right hand side of \eqref{e_error_dtq1} can be bounded by
\begin{equation}\label{e_error_dtq5}
\aligned
-\frac{ d_te_q^{n+1} }{\rm{exp}( -\frac{ t^{n+1} } {T} )} &\left( \textbf{u}^n\cdot\nabla ( 
 \textbf{u}(t^{n+1})-\textbf{u}^n), \textbf{u}(t^{n+1}) \right)\\
 -\frac{ d_te_q^{n+1} }{\rm{exp}( -\frac{ t^{n+1} } {T} )} &\left( ( 
 \textbf{u}(t^{n+1})-\textbf{u}^n)\cdot \nabla \textbf{u}(t^{n+1}), \textbf{u}(t^{n+1}) \right)\\
\leq &\frac{1}{6}|d_te_q^{n+1}|^2+C \|\tilde{e}_{\textbf{u}}^{n+1} \|^2+C(\Delta t)^2.
\endaligned
\end{equation}
The last term on the right hand side of \eqref{e_error_dtq1} can be bounded by
\begin{equation}\label{e_error_dtq6}
\aligned
\textbf{R}_{q}^{n+1}d_te_q^{n+1} \leq \frac{1}{6} |d_te_q^{n+1}|^2+C\Delta t\int_{t^n}^{t^{n+1} } \|q_{tt} \|^2dt.
\endaligned
\end{equation}
Finally combining \eqref{e_error_dtq1} with \eqref{e_error_dtq2}-\eqref{e_error_dtq5}  results in 
\begin{equation}\label{e_error_dtq7}
\aligned
 &| d_te_q^{n+1}|^2+ \frac{|e_q^{n+1}|^2-|e_q^n|^2}{2T\Delta t}+\frac{|e_q^{n+1}-e_q^n|^2}{2T\Delta t}\\
 \leq &\frac{1}{2}|d_te_q^{n+1}|^2+C \| \nabla \tilde{e}_{\textbf{u}}^{n+1} \|^2+C \|\tilde{e}_{\textbf{u}}^{n+1} \|^2+C\Delta t\int_{t^n}^{t^{n+1} } \|q_{tt} \|^2dt+C(\Delta t)^2.
\endaligned
\end{equation}
Multiplying \eqref{e_error_dtq7} by $2T\Delta t$ and summing up for $n$ from $0$ to $m$, we  obtain
\begin{equation*}\label{e_error_dtq8}
\aligned
 T\Delta t\sum\limits_{n=0}^{m}&| d_te_q^{n+1}|^2+ |e_q^{m+1}|^2 \\
 \leq &C\Delta t\sum\limits_{n=0}^{m} \| \nabla \tilde{e}_{\textbf{u}}^{n+1} \|^2 +C \Delta t\sum\limits_{n=0}^{m}\|\tilde{e}_{\textbf{u}}^{n+1} \|^2 
 + C(\Delta t)^2.
\endaligned
\end{equation*}
Combining the above with \eqref{e_error24}, we  obtain the desired result.
\end{proof}

\subsection{Error estimates for the pressure}
The main result in this subsection is the following error estimate for the pressure which requires additional regularities.
\begin{theorem}\label{the: error_estimate_final3}
Assuming $\textbf{u}\in H^3(0,T;\textbf{L}^2(\Omega))\bigcap H^1(0,T;\textbf{H}^2_0(\Omega))\bigcap W^{1,\infty}(0,T;W^{1,\infty}(\Omega))$, $p\in H^2(0,T;H^1(\Omega))$, 
 then for the first-order  scheme \eqref{e_model_semi1}-\eqref{e_model_semi4}, 
 we have
\begin{equation}\label{e_error_p_L2}
\aligned
&\Delta t\sum\limits_{n=0}^m\|e_{p}^{n+1}\|^2_{L^2(\Omega)/R}  \leq C(\Delta t)^2,  \ \ \ \forall \ 0\leq m\leq N-1,
\endaligned
\end{equation}
where $C$ is a positive constant  independent of $\Delta t$.
\end{theorem}

\begin{proof}
 In order to prove the above results, we need to first establish an estimate on $ \| e_{\textbf{u}}^{n+1}-e_{\textbf{u}}^n\|$.    

Adding \eqref{e_error3}  and \eqref{e_error4} leads to 
\begin{equation}\label{e_error_p_add13}
\aligned
\frac{e_{\textbf{u}}^{n+1}-e_{\textbf{u}}^n}{\Delta t}-\nu\Delta\tilde{e}_{\textbf{u}}^{n+1}
&=\frac{q(t^{n+1})}{ \exp( -\frac{ t^{n+1} } {T} ) }(\textbf{u}(t^{n+1})\cdot \nabla)\textbf{u}(t^{n+1})\\
& 
-\frac{q^{n+1}}{ \exp( -\frac{ t^{n+1} } {T}  ) }\textbf{u}^{n}\cdot \nabla\textbf{u}^{n}
-\nabla e_p^{n+1}+\textbf{R}_{\textbf{u}}^{n+1}.
\endaligned
\end{equation}
Define $d_{tt} e_{\textbf{u} }^{n+1}=\frac{d_t e_{\textbf{u} }^{n+1}-d_t e_{\textbf{u} }^n}{\Delta t}$. Then  taking the  difference of two consecutive steps in \eqref{e_error_p_add13}, we have
\begin{equation}\label{e_error_p_add14}
\aligned
&d_{tt} e_{\textbf{u} }^{n+1}-\nu\Delta d_t\tilde{e}_{\textbf{u}}^{n+1}
=d_t\textbf{R}_{\textbf{u}}^{n+1}-\nabla d_te_p^{n+1}+\sum\limits_{i=1}^{3} S_i,
\endaligned
\end{equation}
where 
\begin{equation}\label{e_error_p_add15}
\aligned
S_1=&d_t(q^{n+1}\exp(\frac{t^{n+1}}{T})) ( \textbf{u}(t^{n+1})-\textbf{u}^{n} )\cdot \nabla \textbf{u}(t^{n+1}) \\
&+q^n\exp(\frac{t^{n}}{T}) ( d_t\textbf{u}(t^{n+1})-d_t\textbf{u}^{n} )\cdot \nabla \textbf{u}(t^{n+1}) \\
&+q^n\exp(\frac{t^{n}}{T})  ( \textbf{u}(t^{n})-\textbf{u}^{n-1} )\cdot \nabla d_t\textbf{u}(t^{n+1}),
\endaligned
\end{equation}

\begin{equation}\label{e_error_p_add16}
\aligned
S_2=&d_t(q^{n+1}\exp(\frac{t^{n+1}}{T})) \textbf{u}^{n}\cdot \nabla (\textbf{u}(t^{n+1})-\textbf{u}^{n} )\\
&+q^n\exp(\frac{t^{n}}{T}) d_te_\textbf{u}^{n}\cdot \nabla (\textbf{u}(t^{n+1})-\textbf{u}^{n} )  \\
&+q^n\exp(\frac{t^{n}}{T}) d_t\textbf{u}(t^{n})\cdot \nabla (\textbf{u}(t^{n+1})-\textbf{u}^{n} )\\
&+q^n\exp(\frac{t^{n}}{T}) \textbf{u}^{n-1}\cdot \nabla (d_t\textbf{u}(t^{n+1})-d_t\textbf{u}^{n} ),
\endaligned
\end{equation}
 and  
\begin{equation}\label{e_error_p_add17}
\aligned
S_3=&-d_t(e_q^{n+1} \exp(\frac{t^{n+1}}{T}) ) \textbf{u}(t^{n+1})\cdot \nabla \textbf{u}(t^{n+1})\\
&-e_q^{n} \exp(\frac{t^{n}}{T}) d_t\textbf{u}(t^{n+1})\cdot \nabla \textbf{u}(t^{n+1}) \\
&-e_q^{n} \exp(\frac{t^{n}}{T})  \textbf{u}(t^{n})\cdot \nabla d_t\textbf{u}(t^{n+1}).
\endaligned
\end{equation}
Taking the inner product of \eqref{e_error_p_add14} with $d_t \tilde{e}_{\textbf{u} }^{n+1}$, we find
\begin{equation}\label{e_error_p_add18}
\aligned
&(d_{tt} e_{\textbf{u} }^{n+1},d_t \tilde{e}_{\textbf{u} }^{n+1}) +\nu \|\nabla d_t\tilde{e}_{\textbf{u}}^{n+1}\|^2
=(d_t\textbf{R}_{\textbf{u}}^{n+1},d_t \tilde{e}_{\textbf{u} }^{n+1}) \\
& \ \ \ \ \ \ \ 
-(\nabla d_te_p^{n+1},d_t \tilde{e}_{\textbf{u} }^{n+1})+\sum\limits_{i=1}^{3} (S_i,d_t \tilde{e}_{\textbf{u} }^{n+1} ).
\endaligned
\end{equation}
For the first term on the left hand side, we have
\begin{equation}\label{e_error_p_add20}
\aligned
&(d_{tt} e_{\textbf{u} }^{n+1},d_t \tilde{e}_{\textbf{u} }^{n+1})= 
\frac{ \|d_te_{\textbf{u} }^{n+1}\|^2- \|d_te_{\textbf{u} }^{n}\|^2}{2\Delta t}
+\frac{ \| d_te_{\textbf{u} }^{n+1}-d_te_{\textbf{u} }^n\|^2}{2\Delta t}.
\endaligned
\end{equation}
We bound the terms on the right hand side as follows.

\begin{equation}\label{e_error_p_add21}
\aligned
&(d_t\textbf{R}_{\textbf{u}}^{n+1},d_t \tilde{e}_{\textbf{u} }^{n+1}) \leq \frac{\nu}{8} \|\nabla d_t\tilde{e}_{\textbf{u}}^{n+1}\|^2+C\Delta t\int_{t^{n-1}}^{t^{n+1}} \| \textbf{u}_{ttt}(t) \|^2dt.
\endaligned
\end{equation}
The second term on the right hand of \eqref{e_error_p_add18} can be transformed into 
\begin{equation}\label{e_error_p_add22}
\aligned
&-(\nabla d_te_p^{n+1},d_t \tilde{e}_{\textbf{u} }^{n+1}) =-(\nabla d_te_p^{n},d_t \tilde{e}_{\textbf{u} }^{n+1})  - (\nabla ( d_te_p^{n+1}-d_te_p^{n} ),d_t \tilde{e}_{\textbf{u} }^{n+1}) .
\endaligned
\end{equation}
Since we can derive from \eqref{e_error4} that
\begin{equation}\label{e_error_p_add19}
\aligned
&d_t\tilde{e}_{\textbf{u}}^{n+1}=d_te_{\textbf{u}}^{n+1} +\nabla (p^{n+1}-2p^n+p^{n-1}).
\endaligned
\end{equation}
The first term on the right hand of \eqref{e_error_p_add22} can be estimated by
\begin{equation}\label{e_error_p_add22_R1}
\aligned
&-(\nabla d_te_p^{n},d_t \tilde{e}_{\textbf{u} }^{n+1}) =-(\nabla d_te_p^{n}, \nabla (p^{n+1}-2p^n+p^{n-1}) ) \\
=&-\Delta t(\nabla d_te_p^{n}, \nabla ( d_te_p^{n+1}-d_te_p^n ) )- (\Delta t)^2 \left( \nabla d_te_p^{n}, \frac{1}{ (\Delta t)^2}\nabla (p(t^{n+1})-2p(t^n)+p(t^{n-1}) ) \right) \\
\leq &-\frac {\Delta t} 2 ( \| \nabla d_te_p^{n+1} \|^2- \| \nabla d_te_p^n \|^2- \|\nabla d_te_p^{n+1}-\nabla d_te_p^n \|^2) \\
&+ (\Delta t)^2 \|\nabla d_te_p^{n} \|^2+C\Delta t \int_{t^{n-1}}^{t^{n+1}} \| \nabla p_{tt} \|^2dt.
\endaligned
\end{equation}
The second term on the right hand of \eqref{e_error_p_add22} can be bounded by
\begin{equation}\label{e_error_p_add22_R2}
\aligned
&-(\nabla ( d_te_p^{n+1}-d_te_p^{n} ),d_t \tilde{e}_{\textbf{u} }^{n+1}) =- \left( \nabla  ( d_te_p^{n+1}-d_te_p^{n} ) , \nabla (p^{n+1}-2p^n+p^{n-1}) \right) \\
=&- \Delta t \left(\nabla ( d_te_p^{n+1}-d_te_p^{n} ), \frac{1}{ \Delta t }\nabla (p(t^{n+1})-2p(t^n)+p(t^{n-1}) ) \right)  \\
&- \Delta t \left( \nabla ( d_te_p^{n+1}-d_te_p^{n} ), \nabla ( d_te_p^{n+1}-d_te_p^n ) \right) \\
\leq & -\frac {\Delta t} 2 \|\nabla d_te_p^{n+1}-\nabla d_te_p^n \|^2 +C (\Delta t)^2 \int_{t^{n-1}}^{t^{n+1}} \| \nabla p_{tt} \|^2dt.
\endaligned
\end{equation}
Recalling \eqref{e_estimate for trilinear form},  \eqref{e_L2 norm of u and q} and \eqref{e_error_dtq3} and using Young inequality, we have
\begin{equation}\label{e_error_p_add23}
\aligned
&(S_1, d_t \tilde{e}_{\textbf{u} }^{n+1}) \leq \frac{\nu}{8} \|\nabla d_t\tilde{e}_{\textbf{u}}^{n+1}\|^2+C\| d_te_{\textbf{u}}^n\|^2+C \| e_{\textbf{u}}^n \|^2+C \| e_{\textbf{u}}^{n-1} \|^2+C(\Delta t)^2,
\endaligned
\end{equation}

\begin{equation}\label{e_error_p_add24}
\aligned
(S_2, d_t \tilde{e}_{\textbf{u} }^{n+1}) \leq &C | \left( d_te_\textbf{u}^{n}\cdot \nabla e_{\textbf{u} }^n, d_t \tilde{e}_{\textbf{u} }^{n+1} \right) | +C | \left( \textbf{u}^{n-1} \cdot \nabla d_te_{\textbf{u} }^n, d_t \tilde{e}_{\textbf{u} }^{n+1} \right) | \\
& +  \frac{\nu}{32} \|\nabla d_t\tilde{e}_{\textbf{u}}^{n+1}\|^2+C \| e_{\textbf{u}}^n \|^2_1+C(\Delta t)^2 \\
 \leq &C | \left( d_te_\textbf{u}^{n}\cdot \nabla e_{\textbf{u} }^n, d_t \tilde{e}_{\textbf{u} }^{n+1} \right) | +C | \left( e_{ \textbf{u} }^{n-1} \cdot \nabla d_te_{\textbf{u} }^n, d_t \tilde{e}_{\textbf{u} }^{n+1} \right) | \\
& +  \frac{\nu}{16} \|\nabla d_t\tilde{e}_{\textbf{u}}^{n+1}\|^2+C \| e_{\textbf{u}}^n \|^2_1 +C\| d_te_{\textbf{u}}^n\|^2 +C(\Delta t)^2 \\
\leq & C \| d_te_\textbf{u}^{n} \|^{1/2} \| \nabla d_t \tilde{e}_\textbf{u}^{n} \|^{1/2} \|e_{\textbf{u} }^n\|^{1/2} \|e_{\textbf{u} }^n\|^{1/2}_1\| \nabla d_t \tilde{e}_{\textbf{u} }^{n+1} \| \\
&+C \| e_{ \textbf{u} }^{n-1} \|^{1/2} \| e_{ \textbf{u} }^{n-1} \|^{1/2}_1\| d_te_{\textbf{u} }^n\|^{1/2} \| \nabla d_t \tilde{e}_{\textbf{u} }^n\|^{1/2} \| \nabla d_t \tilde{e}_{\textbf{u} }^{n+1} \|    \\
& +  \frac{\nu}{16} \|\nabla d_t\tilde{e}_{\textbf{u}}^{n+1}\|^2+C \| e_{\textbf{u}}^n \|^2_1 +C\| d_te_{\textbf{u}}^n\|^2 +C(\Delta t)^2 \\
\leq &\frac{\nu}{8} \|\nabla d_t\tilde{e}_{\textbf{u}}^{n+1}\|^2+ C ( \| e_{ \textbf{u} }^{n-1} \|^{2}_0 \| e_{ \textbf{u} }^{n-1} \|^{2}_1+\| e_{ \textbf{u} }^{n} \|^{2}_0 \| e_{ \textbf{u} }^{n} \|^{2}_1 ) \|\nabla d_t\tilde{e}_{\textbf{u}}^{n}\|^2 \\
&+C\| d_te_{\textbf{u}}^n\|^2+C \| e_{\textbf{u}}^n \|^2_1+C(\Delta t)^2,
\endaligned
\end{equation}
and 
\begin{equation}\label{e_error_p_add25}
\aligned
&(S_3, d_t \tilde{e}_{\textbf{u} }^{n+1}) \leq \frac{\nu}{8} \|\nabla d_t\tilde{e}_{\textbf{u}}^{n+1}\|^2+C|d_te_q^{n+1}|^2+C|e_q^{n+1}|^2+C(\Delta t)^2.
\endaligned
\end{equation}
Then combining \eqref{e_error_p_add18} with \eqref{e_error_p_add19}-\eqref{e_error_p_add25}, we have
\begin{equation}\label{e_error_p_add26}
\aligned
&\frac{ \|d_te_{\textbf{u} }^{n+1}\|^2- \|d_te_{\textbf{u} }^{n}\|^2}{2\Delta t}
+\frac{ \| d_te_{\textbf{u} }^{n+1}-d_te_{\textbf{u} }^n\|^2}{2\Delta t} \\
&+\frac \nu 2 \|\nabla d_t\tilde{e}_{\textbf{u}}^{n+1}\|^2 +\frac {\Delta t} 2 ( \|\nabla d_te_p^{n+1} \|^2- \|\nabla d_te_p^n \|^2+ \|\nabla d_te_p^{n+1}-\nabla d_te_p^n \|^2) \\
\ \ \ \leq & (\Delta t)^2 \|\nabla d_te_p^{n} \|^2 + C ( \| e_{ \textbf{u} }^{n-1} \|^{2}_0 \| e_{ \textbf{u} }^{n-1} \|^{2}_1+\| e_{ \textbf{u} }^{n} \|^{2}_0 \| e_{ \textbf{u} }^{n} \|^{2}_1 ) \|\nabla d_t\tilde{e}_{\textbf{u}}^{n}\|^2 \\
&+C\| d_te_{\textbf{u}}^n\|^2+C \| e_{\textbf{u}}^n \|_1^2+C \| e_{\textbf{u}}^{n-1} \|^2 +C|d_te_q^{n+1}|^2 \\
&+C|e_q^{n+1}|^2+C(\Delta t)^2.
\endaligned
\end{equation}
Recalling Theorem \ref{the: error_estimate_final}, we have 
$$\|\nabla d_t\tilde{e}_{\textbf{u}}^{n}\|^2 \leq  (\Delta t)^{-2} \|\nabla \tilde{e}_{\textbf{u}}^{n}\|^2 \leq C (\Delta t)^{-1}, \ \forall 1\leq n\leq N.$$
Then multiplying \eqref{e_error_p_add26} by $2\Delta t$, summing up for $n$ from $1$ to $m$ and applying the discrete Gronwall lemma  \ref{lem: gronwall2}, we can obtain
\begin{equation}\label{e_error_p_add27}
\aligned
 &\|d_te_{\textbf{u} }^{m+1}\|^2+(\Delta t)^2 \|\nabla d_te_p^{m+1} \|^2 \\
 \leq &\|d_te_{\textbf{u} }^{1}\|^2+(\Delta t)^3 \sum\limits_{n=1}^m \|\nabla d_te_p^{n+1} \|^2+
 (\Delta t)^2 \|\nabla d_te_p^{1} \|^2
+ C(\Delta t)^2.
\endaligned
\end{equation}
It remains to estimate $\|d_te_{\textbf{u} }^{1}\|^2$ and $ (\Delta t)^2 \|\nabla d_te_p^{1} \|^2$. Using \eqref{e_error3}, we have 
\begin{equation}\label{e_error_p_add28}
\aligned
 \tilde{e}_{\textbf{u}}^{1}-\nu\Delta t \tilde{e}_{\textbf{u}}^1 =&\Delta t \frac{q(t^{1})}{ \exp(  -\frac{t^{1}}{T} ) }(\textbf{u}(t^{1})\cdot \nabla)\textbf{u}(t^{1})- \Delta t \frac{q^{1}}{ \exp(  -\frac{t^{1}}{T} )  }\textbf{u}^{0}\cdot \nabla\textbf{u}^{0} \\
 &-\Delta t\nabla (p^0-p(t^{1}))+\Delta t \textbf{R}_{\textbf{u}}^{1},
\endaligned
\end{equation}
Taking the inner product of \eqref{e_error_p_add28} with $\tilde{e}_{\textbf{u}}^1$ leads to
\begin{equation}\label{e_error_p_add29}
\aligned
 \| \tilde{e}_{\textbf{u}}^{1} \|^2+\nu\Delta t \| \nabla \tilde{e}_{\textbf{u}}^1 \|^2 
 =& \Delta t\left( \frac{q(t^{1})}{ \exp(  -\frac{t^{1}}{T} ) }(\textbf{u}(t^{1})\cdot \nabla)\textbf{u}(t^{1})- \frac{q^{1}}{ \exp(  -\frac{t^{1}}{T} )  }\textbf{u}^{0}\cdot \nabla\textbf{u}^{0}, \tilde{e}_{\textbf{u}}^1 \right) \\
 &-\Delta t (\nabla (p(t^0)-p(t^{1})), \tilde{e}_{\textbf{u}}^1)
 +\Delta t (\textbf{R}_{\textbf{u}}^{1}, \tilde{e}_{\textbf{u}}^1) \\
 \leq &\frac{1}{2}\| \tilde{e}_{\textbf{u}}^{1} \|^2+C(\Delta t)^4,
\endaligned
\end{equation}
from which we obtain 
$$\|d_te_{\textbf{u} }^{1}\|^2 \leq  \|d_t\tilde{e}_{\textbf{u} }^{1}\|^2 =(\Delta t)^{-2}\| \tilde{e}_{\textbf{u}}^{1} \|^2\leq C(\Delta t)^2.$$
We can derive from \eqref{e_error4} with $n=1$ that
\begin{equation}\label{e_error_p_add30}
\aligned
& (\Delta t)^2\|\nabla d_te_p^{1} \|^2 \leq (\Delta t)^{-2}( \| e_{\textbf{u}}^{1} \|^2+ \| \tilde{e}_{\textbf{u}}^{1} \|^2) + (\Delta t)^2\| \nabla d_tp(t^{1}) \|^2 \leq C (\Delta t)^2.
\endaligned
\end{equation}
Combing the above estimates into  \eqref{e_error_p_add27}, we finally obtain  
\begin{equation}\label{e_error_p_add31}
\aligned
 &\|d_te_{\textbf{u} }^{m+1}\|^2+(\Delta t)^2 \|\nabla d_te_p^{m+1} \|^2 \leq 
C(\Delta t)^2,
\endaligned
\end{equation}
which implies in particular
\begin{equation}\label{error_dtu}
 \| e_{\textbf{u}}^{n+1}-e_{\textbf{u}}^n\|\le C (\Delta t)^2.
 \end{equation}

We are now in position to prove the pressure estimate. 

\begin{equation}\label{e_error_p_add13******}
\aligned
\frac{e_{\textbf{u}}^{n+1}-e_{\textbf{u}}^n}{\Delta t}-\nu\Delta\tilde{e}_{\textbf{u}}^{n+1}
&=\frac{q(t^{n+1})}{ \exp( -\frac{ t^{n+1} } {T} ) }(\textbf{u}(t^{n+1})\cdot \nabla)\textbf{u}(t^{n+1})\\
& 
-\frac{q^{n+1}}{ \exp( -\frac{ t^{n+1} } {T}  ) }\textbf{u}^{n}\cdot \nabla\textbf{u}^{n}
-\nabla e_p^{n+1}+\textbf{R}_{\textbf{u}}^{n+1}.
\endaligned
\end{equation}

 Taking the inner product of \eqref{e_error_p_add13} with $\textbf{v}\in \textbf{H}^1_0(\Omega)$, we have
\begin{equation}\label{e_error_p10}
\aligned
&(\nabla e_p^{n+1},\textbf{v})=-(\frac{e_{\textbf{u}}^{n+1}-e_{\textbf{u}}^n}{\Delta t},\textbf{v})+\nu(\Delta\tilde{e}_{\textbf{u}}^{n+1},\textbf{v})+(\textbf{R}_{\textbf{u}}^{n+1},\textbf{v})\\
&+ \left( \frac{q(t^{n+1})}{ \exp(  -\frac{ t^{n+1} } {T} ) } \textbf{u}(t^{n+1})\cdot \nabla \textbf{u}(t^{n+1})-\frac{q^{n+1}}{ \exp(  -\frac{ t^{n+1} } {T} ) } \textbf{u}^{n}\cdot \nabla\textbf{u}^{n},\textbf{v}\right).
\endaligned
\end{equation}
Taking notice of the fact that 
\begin{equation}\label{e_error_p11}
\aligned
&\| p\|_{L^2(\Omega)/\mathbb{R}} \leq \sup_{\textbf{v} \in \textbf{H}^1_0(\Omega)} \frac{
(\nabla e_p^{n+1},\textbf{v}) }{ \|\nabla \textbf{v} \| }.
\endaligned
\end{equation} 
By using \eqref{e_error8}-\eqref{e_error10} and \eqref{e_error_dtq3}, we can derive that, for all $\textbf{v}\in \textbf{H}^1_0(\Omega)$, 
\begin{equation}\label{e_error_p10_R1}
\aligned
& \left( \frac{q(t^{n+1})}{ \exp(  -\frac{ t^{n+1} } {T} ) } \textbf{u}(t^{n+1})\cdot \nabla \textbf{u}(t^{n+1})-\frac{q^{n+1}}{ \exp(  -\frac{ t^{n+1} } {T} ) } \textbf{u}^{n}\cdot \nabla\textbf{u}^{n}, \textbf{v} \right) \\
=&\frac{q(t^{n+1}) }{ \exp(  -\frac{t^{n+1}}{T} ) }\left( ( \textbf{u}(t^{n+1})-\textbf{u}^{n} )\cdot \nabla \textbf{u}(t^{n+1}),  \textbf{v} \right)\\
&+\frac{q(t^{n+1}) }{ \exp(  -\frac{t^{n+1}}{T} ) }\left(  \textbf{u}^{n}\cdot \nabla (\textbf{u}(t^{n+1})-\textbf{u}^{n} ),  \textbf{v} \right)\\
&-\frac{e_q^{n+1}}{ \exp(  -\frac{t^{n+1}}{T} ) }\left((\textbf{u}^n \cdot \nabla)\textbf{u}^n,  \textbf{v} \right) \\
\leq & C(\|e_{ \textbf{u} }^{n}\| +\| \nabla \tilde{e}_{\textbf{u}}^n\| +\| \int_{t^n}^{t^{n+1}}\textbf{u}_tdt \|_0+ |e_q^{n+1}| )  \|\nabla \textbf{v} \| .
\endaligned
\end{equation}
Hence thanks to Theorem \ref{the: error_estimate_final} and \eqref{e_error13}, \eqref{e_error_p_add31}, we can derive from the above that 
\begin{equation}\label{e_error_p12}
\aligned
&\Delta t \sum\limits_{n=0}^m \|e_p^{n+1}\|^2_{L^2(\Omega)/\mathbb{R}} \leq 
C\Delta t \sum\limits_{n=0}^m \left( \| d_te_{\textbf{u}}^{n+1}\|^2+ \| \nabla \tilde{e}_{\textbf{u}}^{n+1} \|^2 \right. \\
&\ \ \ \ \ 
\left. +  \| \nabla \tilde{e}_{\textbf{u}}^{n} \|^2+\|e_{\textbf{u}}^n\|^2 + |e_q^{n+1}|^2 \right) \\
&\ \ \ \ \  +C (\Delta t)^2 \int_{t^0}^{t^{m+1}} (\|\textbf{u}_t\|_1^2+\|\textbf{u}_{tt}\|_{-1}^2 ) dt   \leq  C(\Delta t)^2.
\endaligned
\end{equation}
The proof  is complete.
\end{proof}

 \section{Numerical experiments} 
In this section, we carry out some numerical experiments to verify the accuracy of the first- and second-order SAV schemes with pressure correction for the Navier-Stokes equations. In all examples below, we take $\Omega=(0,1)\times(0,1)$, $T=1$, $\mu=0.1$, and  the spatial discretization is based on the MAC scheme on the staggered grid with $N_x=N_y=250$ so that the spatial discretization error is negligible compared to the time discretization error for the time steps used in the experiments.

 \begin{table}[htbp]
\renewcommand{\arraystretch}{1.1}
\small
\centering
\caption{Errors and convergence rates for example 1 with the first-order scheme \eqref{e_model_semi1}-\eqref{e_model_semi4} }\label{table1_example1}
\begin{tabular}{p{1.2cm}p{1.5cm}p{1.5cm}p{1.5cm}p{1.5cm}p{1.5cm}p{1.5cm}}\hline
$\Delta t$    &$\|e_{\textbf{u}}\|_{l^{\infty}}$    &Rate &$\|e_p\|_{l^2}$   &Rate  
&$\|e_{q}\|_{\infty}$    &Rate   \\ \hline
$\frac{1}{10}$    &5.77E-3                & ---    &2.20E-2         &---  &2.26E-2         &---\\
$\frac{1}{20}$    &2.25E-3                &1.36    &1.06E-2         &1.06 &1.02E-2         &1.14\\
$\frac{1}{40}$    &1.04E-3                &1.11     &5.13E-3         &1.04 &4.87E-3         &1.07\\
$\frac{1}{80}$    &5.01E-4                &1.05    &2.54E-3         &1.01   &2.37E-3         &1.04\\
\hline
\end{tabular}
\end{table}

\begin{table}[htbp]
\renewcommand{\arraystretch}{1.1}
\small
\centering
\caption{Errors and convergence rates for example 1 with the second-order scheme \eqref{e_model_semi_second1}-\eqref{e_model_semi_second4} }\label{table2_example1}
\begin{tabular}{p{1.2cm}p{1.5cm}p{1.5cm}p{1.5cm}p{1.5cm}p{1.5cm}p{1.5cm}}\hline
$\Delta t$    &$\|e_{\textbf{u}}\|_{l^{\infty}}$    &Rate &$\|e_p\|_{l^2}$   &Rate  
&$\|e_{q}\|_{\infty}$    &Rate   \\ \hline
$\frac{1}{10}$    &1.99E-3                & ---    &7.83E-3         &---  &4.69E-3         &---\\
$\frac{1}{20}$    &5.25E-4                &1.92    &2.47E-3         &1.66 &1.24E-3         &1.92\\
$\frac{1}{40}$    &1.36E-4                &1.95     &7.20E-4         &1.78 &3.17E-4         &1.97\\
$\frac{1}{80}$    &3.95E-5                &1.78    &1.99E-4         &1.85   &7.97E-5         &1.99\\
\hline
\end{tabular}
\end{table}

\begin{table}[htbp]
\renewcommand{\arraystretch}{1.1}
\small
\centering
\caption{Errors and convergence rates for example 2 with the first-order scheme \eqref{e_model_semi1}-\eqref{e_model_semi4} }\label{table1_example2}
\begin{tabular}{p{1.2cm}p{1.5cm}p{1.5cm}p{1.5cm}p{1.5cm}p{1.5cm}p{1.5cm}}\hline
$\Delta t$    &$\|e_{\textbf{u}}\|_{l^{\infty}}$    &Rate &$\|e_p\|_{l^2}$   &Rate  
&$\|e_{q}\|_{\infty}$    &Rate   \\ \hline
$\frac{1}{10}$    &1.14E-2                & ---    &2.13E-2         &---  &2.03E-2         &---\\
$\frac{1}{20}$    &5.08E-3                &1.17    &1.07E-2         &0.99 &9.44E-3         &1.11\\
$\frac{1}{40}$    &2.46E-3                &1.05     &5.30E-3         &1.01 &4.61E-3         &1.03\\
$\frac{1}{80}$    &1.23E-3                &1.00    &2.63E-3         &1.01   &2.30E-3         &1.01\\
\hline
\end{tabular}
\end{table}

\begin{table}[htbp]
\renewcommand{\arraystretch}{1.1}
\small
\centering
\caption{Errors and convergence rates for example 2 with the second-order scheme \eqref{e_model_semi_second1}-\eqref{e_model_semi_second4} }\label{table2_example2}
\begin{tabular}{p{1.2cm}p{1.5cm}p{1.5cm}p{1.5cm}p{1.5cm}p{1.5cm}p{1.5cm}}\hline
$\Delta t$    &$\|e_{\textbf{u}}\|_{l^{\infty}}$    &Rate &$\|e_p\|_{l^2}$   &Rate  
&$\|e_{q}\|_{\infty}$    &Rate   \\ \hline
$\frac{1}{10}$    &3.95E-3                & ---    &5.95E-3         &---  &1.82E-3         &---\\
$\frac{1}{20}$    &1.06E-3                &1.90    &1.66E-3         &1.84 &4.09E-4         &2.16\\
$\frac{1}{40}$    &2.77E-4                &1.94     &4.51E-4         &1.88 &9.82E-5         &2.06\\
$\frac{1}{80}$    &8.09E-5                &1.78    &1.21E-4         &1.89   &2.42E-5         &2.02\\
\hline
\end{tabular}
\end{table}

{\bf Example 1}. The right hand side of the equations are computed according to the analytic solution given as below:\\
\begin{equation*}
\aligned
\begin{cases}
p(x,y,t)=\sin(t)(\sin(\pi y)-2/\pi),\\
u_1(x,y,t)= \sin(t)\sin^2(\pi x)\sin(2\pi y),\\
u_2(x,y,t)=- \sin(t)\sin(2\pi x)\sin^2(\pi y).
\end{cases}
\endaligned
\end{equation*}

{\bf Example 2}. The right hand side of the equations are computed according to the analytic solution given as below:\\
\begin{equation*}
\aligned
\begin{cases}
p(x,y,t)=t^2(x-0.5),\\
u_1(x,y,t)=-128t^2x^2(x-1)^2y(y-1)(2y-1),\\
u_2(x,y,t)=128t^2y^2(y-1)^2x(x-1)(2x-1).
\end{cases}
\endaligned
\end{equation*}
Numerical results for Examples 1 and 2 with first- and second-order schemes are presented in Tables \ref{table1_example1}-\ref{table2_example2}. We observe that the results for the first-order scheme are consistent with the error estimates in Theorems  \ref{the: error_estimate_final} and \ref{the: error_estimate_final3}. While second-order convergence rates for the velocity and SAV variable in $L^\infty$ norm, and nearly second-order convergence rates for the pressure in $L^2$ norm  were observed for the second-order scheme.

\begin{table}[htbp]
\renewcommand{\arraystretch}{1.1}
\small
\centering
\caption{Errors and convergence rates for example 1 with the first-order scheme \eqref{e_model_semi_original1}-\eqref{e_model_semi_original4} }\label{table3_example1}
\begin{tabular}{p{1.3cm}p{2.0cm}p{1.5cm}p{2.0cm}p{1.5cm}p{1.5cm}p{1.5cm}}\hline
$\Delta t$    &$\|e_{\textbf{u}}\|_{l^{\infty}}$    &Rate &$\|e_p\|_{l^2}$   &Rate   \\ \hline
$\frac{1}{10}$    &5.75E-3                & ---    &2.10E-2         &---  \\
$\frac{1}{20}$    &2.24E-3                &1.36    &9.55E-3         &1.13\\
$\frac{1}{40}$    &1.04E-3                &1.11     &4.46E-3         &1.10 \\
$\frac{1}{80}$    &5.01E-4                &1.05    &2.17E-3         &1.04 \\
\hline
\end{tabular}
\end{table}

\begin{table}[htbp]
\renewcommand{\arraystretch}{1.1}
\small
\centering
\caption{Errors and convergence rates for example 2 with the first-order scheme \eqref{e_model_semi_original1}-\eqref{e_model_semi_original4} }\label{table3_example2}
\begin{tabular}{p{1.3cm}p{2.0cm}p{1.5cm}p{2.0cm}p{1.5cm}p{1.5cm}p{1.5cm}}\hline
$\Delta t$    &$\|e_{\textbf{u}}\|_{l^{\infty}}$    &Rate &$\|e_p\|_{l^2}$   &Rate   \\ \hline
$\frac{1}{10}$    &1.14E-2                & ---    &1.88E-2         &---  \\
$\frac{1}{20}$    &5.05E-3                &1.17    &8.91E-3         &1.08\\
$\frac{1}{40}$    &2.44E-3                &1.05     &4.30E-3         &1.05 \\
$\frac{1}{80}$    &1.22E-3                &1.00    &2.11E-3         &1.03 \\
\hline
\end{tabular}
\end{table}

As a comparison, we also implemented the following pressure-correction version of the scheme \eqref{e_model_semi_original1}-\eqref{e_model_semi_original4}.

\textbf{Scheme \uppercase\expandafter{\romannumeral 3}:} Find ($\tilde{\textbf{u}}^{n+1}$, $\textbf{u}^{n+1}$,$p^{n+1}$,$q^{n+1}$) by solving
  \begin{numcases}{}
   \frac{\tilde{\textbf{u}}^{n+1}-\textbf{u}^{n}}{\Delta t}+\frac{q^{n+1}}{\sqrt{E(\textbf{u}^{n})+C_0}}\textbf{u}^{n}\cdot \nabla\textbf{u}^{n}
     -\nu\Delta\tilde{\textbf{u}}^{n+1}     +\nabla p^{n}=\textbf{f}^{n+1}, \ \ \tilde{\textbf{u}}^{n+1}|_{\partial \Omega}=0; \label{e_model_semi_original1b}\\
     \frac{\textbf{u}^{n+1}-\tilde{\textbf{u}}^{n+1}}{\Delta t}+\nabla(p^{n+1}-p^n)=0; \label{e_model_semi_original2b}\\
     \nabla\cdot\textbf{u}^{n+1}=0, \ \ \textbf{u}^{n+1}\cdot \textbf{n}|_{\partial \Omega}=0;
 \label{e_model_semi_original3b}\\ 
   2q^{n+1}\frac{q^{n+1}-q^n}{\Delta t}=(\frac{\textbf{u}^{n+1}-\textbf{u}^{n}}{\Delta t}+ 
   \frac{q^{n+1}}{\sqrt{E(\textbf{u}^{n})+C_0}}(\textbf{u}^n\cdot\nabla) \textbf{u}^n,\tilde{\textbf{u}}^{n+1}), \label{e_model_semi_original4b} 
\end{numcases}
where $E(\textbf{u})=\int_{\Omega}\frac{1}{2}|\textbf{u}|^2$ is the total energy. 
Numerical results with the MAC discretization of the above scheme is listed in  Tables \ref{table3_example1} and \ref{table3_example2}. 
It is  observed that the results with this scheme are essentially the same as the results by our new first-order SAV scheme  in Tables \ref{table1_example1} and \ref{table1_example2}. Note that the above scheme requires solving a nonlinear algebraic equation at each time step.

\section{Concluding remarks}

We constructed novel first- and second-order linear and decoupled pressure correction  schemes based on the SAV approach for the Navier-Stokes equations, and proved that they are unconditionally energy stable. 
Compared with the previous version of SAV scheme \eqref{e_model_semi_original1}-\eqref{e_model_semi_original4}, the new schemes possess two distinct advantages: (i)  they are purely linear, {\color{black} eliminating the numerical and theoretical difficulties associated with the nonlinear algebraic equation in \eqref{e_model_semi_original4}, } and (ii) they lead to a much stronger stability result with a uniform bound on  the $L^2$-norm of the numerical solution, which is essential for the error analysis, and enable us to derive optimal error estimates for the first-order scheme without any restriction on the time step. Another main contribution is that we proved unconditional energy stability for the new SAV scheme based on the second-order rotational pressure-correction scheme. 
To the best of the authors' knowledge, these schemes are the first of such kind  for the Navier-Stokes equations with unconditional energy stability while treating the nonlinear term explicitly.

 We only carried out a rigorous error analysis for the  first-order scheme. Due to the rotational form of the pressure correction in the second-order scheme, its error analysis will be much more involved, as indicated by the technicality in its analysis without the nonlinear term \cite{guermond2004error}. We shall leave its analysis for a future endeavor. 
\bibliographystyle{siamplain}
\bibliography{LM_Pre_Correct}

\end{document}